\documentclass[11pt, a4paper]{amsart}
\usepackage{amssymb, array, amsmath, amscd, pdfpages, enumerate, amsthm, setspace,mathtools}

\usepackage[utf8]{inputenc}
\usepackage{amssymb}
\usepackage{bm}
\usepackage{overpic}
\usepackage[colorlinks=true,allcolors=magenta]{hyperref}
\usepackage{xcolor}
\usepackage{graphicx} 
\usepackage{ifthen}
\usepackage[capitalize,noabbrev]{cleveref}

\newcommand{\BNsp}[4]{(#1,#2,#3,#4)}
\newcommand{\BNspStandard}{$\BNsp{U_b}{U}{X}{Y}$}
\newcommand{\Dom}{D}

\DeclareMathOperator{\re}{Re}

\DeclareMathOperator{\diag}{diag}

\newcommand{\mcA}{{\mathfrak{A}}}     
\newcommand{\mcB}{{\mathfrak{B}}}     
\newcommand{\mcC}{{\mathfrak{C}}}     
\newcommand{\mcK}{{\mathfrak{K}}}     

\newcommand*{\C}{{\mathbb{C}}}     
\newcommand*{\R}{{\mathbb{R}}}     
     
\newcommand*{\N}{{\mathbb{N}}}     
\newcommand*{\Hloc}[1]{H^{#1}_{loc}}

\newcommand*{\Lin}{{\mathcal{L}}}   
\newcommand{\ran}{{\mathcal{R}}}   
\renewcommand{\ker}{{\mathcal{N}}}

\newcommand*{\abs}[1]{\lvert#1\rvert}
\newcommand*{\norm}[1]{\lVert#1\rVert}
\newcommand*{\set}[1]{\{#1\}}
\newcommand*{\setm}[2]{\{\,#1\mid#2\,\}}   
\newcommand*{\iprod}[2]{\langle#1,#2\rangle}    
    
\newcommand*{\Norm}[2][default]{\ifthenelse{\equal{#1}{default}}{\left\lVert#2\right\rVert}{\ldelim{#1}{\lVert}#2\rdelim{#1}{\rVert}}}

\newcommand*{\ldelim}[2]{\csname#1l\endcsname#2}   
\newcommand*{\rdelim}[2]{\csname#1r\endcsname#2}   
\newcommand*{\mdelim}[2]{\csname#1m\endcsname#2}   
 
\newcommand*{\Lp}[1][p]{L^{#1}}

\newcommand*{\Lploc}[1][p]{L^{#1}_{\text{loc}}}

  \newcommand{\pmat}[1]{\begin{bmatrix}#1\end{bmatrix}}
\newcommand{\pmatsmall}[1]{\begin{bsmallmatrix}#1\end{bsmallmatrix}}

\newcommand{\eq}[1]{\begin{align*}#1\end{align*}}
\newcommand{\eqn}[1]{\begin{align}#1\end{align}}
\newcommand{\ieq}[1]{$#1$}

\newcommand{\gs}{\sigma}

\newcommand{\gb}{\beta}

\newcommand{\gd}{\delta}
\newcommand{\gl}{\lambda}
\newcommand{\gw}{\omega}

\newcommand{\inv}{^{-1}}
 
\newcommand{\tp}{}

\newcommand{\citel}[2]{\cite[#2]{#1}}

\renewcommand{\pmat}[1]{\begin{bmatrix}#1\end{bmatrix}}
\renewcommand{\pmatsmall}[1]{\begin{bsmallmatrix}#1\end{bsmallmatrix}}

\newcommand{\F}{\mathbb{F}}
\newcommand{\Pt}[1][t]{\mathbf{P}_{#1}}

\renewcommand{\Dom}{D}

\renewcommand{\AA}{\mathfrak{A}}
\newcommand{\BB}{\mathfrak{B}}
\newcommand{\CC}{\mathfrak{C}}
\newcommand{\KK}{\mathfrak{K}}

\renewcommand{\pmat}[1]{\begin{bmatrix}#1\end{bmatrix}}
\renewcommand{\pmatsmall}[1]{\begin{bsmallmatrix}#1\end{bsmallmatrix}}

\renewcommand{\ran}{\textup{Ran}}
\renewcommand{\ker}{\textup{Ker}}

\newcommand{\B}{\Lin}

\newcommand*{\dda}[3][1]{\ifthenelse{\equal{#1}{1}}{\frac{d#3}{d#2}}{\frac{d^{#1}#3}{d#2^{#1}}}}
\newcommand*{\ddb}[2][1]{\ifthenelse{\equal{#1}{1}}{\frac{d}{d#2}}{\frac{d^{#1}}{d#2^{#1}}}}
\newcommand*{\pd}[3][1]{\ifthenelse{\equal{#1}{1}}{\frac{\partial{#2}}{\partial{#3}}}{\frac{\partial^{#1}{#2}}{\partial#3^{#1}}}}
\newcommand*{\pdb}[2][1]{\ifthenelse{\equal{#1}{1}}{\frac{\partial}{\partial{#2}}}{\frac{\partial^{#1}}{\partial#2^{#1}}}}

\newcommand{\zinf}[1][0]{[#1,\infty)}

\newcommand{\dtot}{d_{\mbox{\scriptsize\textit{tot}}}}

\newcommand{\detailTOCHECK}[1]{{\color{gray}#1}}
\renewcommand{\detailTOCHECK}[1]{}

\newtheorem{theorem}{Theorem}[section]
\newtheorem{lemma}[theorem]{Lemma}
\newtheorem{proposition}[theorem]{Proposition}

\newtheorem{assumption}[theorem]{Assumption}
\theoremstyle{definition}
\newtheorem{definition}[theorem]{Definition}
\newtheorem{remark}[theorem]{Remark}

\numberwithin{equation}{section}

\crefname{proposition}{Proposition}{Propositions}
\Crefname{proposition}{Proposition}{Propositions}
\crefname{lemma}{Lemma}{Lemmas}
\Crefname{lemma}{Lemma}{Lemmas}
\crefname{corollary}{Corollary}{Corollaries}
\Crefname{corollary}{Corollary}{Corollaries}
\crefname{definition}{Definition}{Definitions}
\Crefname{definition}{Definition}{Definitions}
\crefname{remark}{Remark}{Remarks}
\Crefname{remark}{Remark}{Remarks}
\crefname{example}{Example}{Examples}
\Crefname{example}{Example}{Examples}
\crefname{assumption}{Assumption}{Assumptions}
\Crefname{assumption}{Assumption}{Assumptions}

\begin{document}

\title[Active Disturbance Rejection Control for BCS]{Active Disturbance Rejection for Boundary Control Systems}

\author[J.-P.~Humaloja]{Jukka-Pekka Humaloja}
\address[J.-P.~Humaloja]{Department of Electrical and Computer
Engineering, Technical University of Crete, Chania, Greece.}
 \email{jhumaloja@tuc.gr}

\author[L.~Paunonen]{Lassi Paunonen}
\address[L.~Paunonen]{Mathematics Research Centre, Tampere University, P.O.~ Box 553, 33101 Tampere, Finland}
 \email{lassi.paunonen@tuni.fi}

\thanks{This work was supported by the Research Council of Finland Grant number 349002.}

\thispagestyle{plain}

\begin{abstract}
We consider stabilisation of abstract boundary control systems and controlled partial differential equations with general unknown input disturbances and unmodeled nonlinearities at the input. We utilise the active disturbance rejection control approach to design a controller which rejects the input disturbance and achieves stability and external well-posedness of the closed-loop system for a class of boundary control systems with collocated inputs and outputs. We apply our main results to design controllers for one-dimensional wave and heat equations.
\end{abstract}

\keywords{
Dynamic stabilisation, unknown input disturbance, active disturbance rejection control, boundary control system, wave equation, heat equation.
}
\subjclass[2020]{%
93C25, 
93C10, 
93D15, 
  35B35 
(93C20,  
47D06, 
47A10)
}

\maketitle

\section{Introduction}
\label{sec:intro}

In this article we study stabilisation 
of infinite-dimensional linear systems and partial differential equations (PDEs) in the situation where the input of the system is corrupted by an unknown disturbance signal and unmodeled nonlinearities.
Our main interest is in controller design for classes of PDE systems with boundary inputs and outputs. Motivated by this we consider control of an abstract boundary control system~\cite{Sal87a,MalSta06} of the form
\begin{subequations}
\label{eq:BCSintro}
\eqn{
\dot x(t) &= \mcA x(t) , \hspace{4cm} x(0)=x_0 \\
\label{eq:BCSintro2}
\mcB x(t) &=  u(t) + d(t)+ \phi_s(x(t)) + \phi_o(y(t)) \\
y(t) &= \mcC x(t).
}
\end{subequations}
Here $x$, $u$, and $y$ are the state, the input and the output of the system. 
In a boundary control system $\BB$ is a boundary trace operator  chosen in such a way that~\eqref{eq:BCSintro2} describes the controlled boundary conditions of the PDE system.
 The function $d$  is the unknown input disturbance while the possibly unknown functions $\phi_s$ and $\phi_o$  represent unmodeled nonlinearities at the input of the system.
As our main result we introduce a dynamic controller based on the \emph{active disturbance rejection control} (ADRC)~\cite{GuoZha16book,FenGuo17a} approach to estimate and reject the ``total disturbance'' $\dtot=d+ \phi_s(x) + \phi_o(y)$ and to drive the state $x$ of the system~\eqref{eq:BCSintro} to zero as $t\to \infty$.

Active disturbance rejection control was first introduced for nonlinear systems~\cite{Han09}, 
but it has also been demonstrated to be a powerful tool in control of PDE systems with boundary control and observation~\cite{ZhoGuo17, FenGuo16a, KelSae23}. In the PDE setting, the estimation and rejection of the unknown total disturbance is based on the simple but very useful observation that if the output $y$ of the system~\eqref{eq:BCSintro} is connected to an inverse of the linear system $(\BB,\AA,\CC)$ as in \cref{fig:DistEst}, then the output of the inverse provides an asymptotic estimate for the input $u+\dtot $ of the original system, thus leading to an estimate $\hat d$ for $\dtot$. In the controller design the total disturbance can then be rejected using an input of the form $u=u_0-\hat d$, where $u_0$ is a control signal used for stabilisation of the system without the disturbance $\dtot$. 

\begin{figure}[h!]
\begin{center}
\includegraphics[width=.7\linewidth]{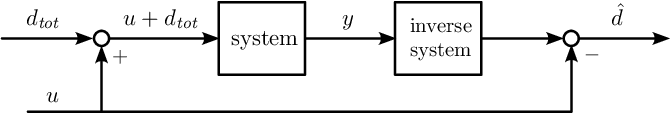}
\end{center}
\caption{The estimation scheme for the total disturbance.}
\label{fig:DistEst}
\end{figure}

The main limitation of the disturbance estimation scheme in \cref{fig:DistEst} is that the inverse $P\inv$ needs to be a well-defined dynamical system. In the case of finite-dimensional linear systems this is only possible if the system has an invertible feed\-through matrix, in which case the input disturbance estimation becomes easy. However, there are several nontrivial PDE systems whose inverses have well-defined dynamics. In particular, constructing the inverse of a PDE system with boundary control and observation is done by simply reinterpreting the boundary conditions defining the input to instead be a measurement at the boundary values of the state, and by changing the boundary traces defining the output to be the new boundary conditions of the systems determined by the values of $y$. 
The same process can be completed in the setting of boundary control systems of the form~\eqref{eq:BCSintro} to construct an estimator 
\begin{subequations}
\label{eq:BCSinvintro}
\eqn{
\dot x_i(t) &= \mcA x_i(t) , \qquad\quad x_i(0)=x_{i0} \\
 \mcC x_i(t)&=y(t) \\
\hat d(t)&= \mcB x_i(t) -u(t)   .
}
\end{subequations}
Roughly stated, if the inverted system $(\CC,\AA,\BB)$ has well-defined and stable dynamics, then the estimation error $\hat d-\dtot $ converges to zero as $t\to \infty$. In this situation $\hat d$ can be used to asymptotically reject the total disturbance $\dtot$ at the input of the system, as well as to approximate $\dtot$ in a \emph{linear} observer for the system~\eqref{eq:BCSintro} for the purpose of dynamic stabilisation.
Our controller in \cref{sec:AbstractController} consists of exactly such an inverted system and a linear Luenberger-type observer for~\eqref{eq:BCSintro}. Finally, a third block is added for pre-stabilisation of~\eqref{eq:BCSintro}
before the inversion.

Active disturbance rejection control has been employed in the control of various PDE models on both one-dimensional and multi-dimensional spatial domains. 
It has been especially actively used in the unknown disturbance rejection and stabilisation of wave equations, both on one-dimensional domains~\cite{GuoJin13,GuoJin15,FenGuo17b,ZhoWei18,MeiZho21} and multi-dimensional domains~\cite{ZhoGuo15,FenGuo16a,ZhoGuo17,FenGuo17a}, and for models including additional ODE dynamics~\cite{Mei22,KelSae23}. 
These existing works differ in the types of boundary conditions and output configurations used in the models, and in whether or not input nonlinearities are considered.
Similarly, ADRC design has been successfully used in the control of Euler--Bernoulli beam equations~\cite{ZhoFen18,ZhoFen21}, a Timoshenko beam equation~\cite{FanXu25}, and Schrödinger equations~\cite{Zho17,JiaLiu20}. 
Active disturbance rejection has also been demonstrated to be a powerful tool in the control of heat equations, which is particularly remarkable because
 the inverses of these systems are characteristically ill-posed. 
In the references on heat equations on one-dimensional domains~\cite{FenGuo17d} and multi-dimensional domains~\cite{ZhoGuo17,FenXu20} the authors demonstrate that 
the lack of well-posedness of the inverse system is naturally compensated by the smoothing effect of the parabolic models.
We take advantage of this same phenomenon in our second main result in \cref{thm:ADRCParab}.

In this article we consider controller design for abstract boundary control systems and the class of PDE models which can be represented within this class. Our main assumption  is that the linear version of~\eqref{eq:BCSintro} (i.e., with $\phi_s=0$ and $\phi_o=0$) is \emph{well-posed} in the sense that for all classical solutions its state $x$ and $y$  depend continuously on the initial state $x_0$ and the input $u+d$ on finite time intervals (see \cref{def:BCSwellposed}). 
Moreover, $\phi_s$ is locally Lipschitz with linear growth and $\phi_o$ is globally Lipschitz. 
In our two main results in \cref{thm:ADRCmain} and \cref{thm:ADRCParab}
the inverted system in the controller is assumed to posses different weaker well-posedness properties.
Our main results establish the existence of classical and generalised solutions for the nonlinear closed-loop system consisting of~\eqref{eq:BCSintro} and the controller and establishes exponential convergence of the disturbance estimation and the controlled state to zero under minimal conditions.
While infinite-dimensional systems theory has had a role in the ADRC design for PDEs~\cite{ZhoGuo15,ZhoWei18,FenXu20}, this methodology has not previously been employed in the context of abstract infinite-dimensional systems.
Our results also strengthen the conclusions of existing ADRC designs for PDE models 
 in the following aspects: 
\begin{itemize}
\item Majority of the earlier references show asymptotic convergence of the disturbance estimation and the state as $t\to \infty$. Our results establish exponentially fast convergence.
\item Our controller is capable of handling boundary nonlinearities, whereas earlier results have only considered the case $\phi_o=0$.
\item We add external inputs to the closed-loop system and establish external incremental well-posedness of the nonlinear closed-loop system.
\item Our results present a controller structure for the full controller and
collects the conditions on the controller parameters to a set of assumptions which can be used efficiently in the controller tuning.
\end{itemize}

We illustrate the use of our results for PDEs 
  in \cref{sec:PDEs} where we design controllers for one-dimensional wave and heat equations.
It is worthwile to note that the active disturbance rejection controller 
establishes the existence of closed-loop solutions under much weaker conditions than the general required conditions for the open loop system~\eqref{eq:BCSintro}, namely, without global Lipschitz property of $\phi_s$, without bounds on the Lipschitz constant of $\phi_o$ (cf.~\citel{TucWei14}{Sec.~7},~\cite{GuiLog19}) and without passivity or monotonicity assumptions (cf.~\cite{HasCal19,HasPau25arxiv}).
Our use of the abstract boundary control systems framework is motivated by its ease of use in controller design for PDEs. Indeed, representation of a PDE with boundary inputs and outputs in the form~\eqref{eq:BCSintro} is relatively easy compared to finding its representation as a regular linear system or a system node, both involving an unbounded control operator $B$. In addition, interpreting the abstract controller as a PDE model is similarly straightforward.

\textbf{Notation.} 
\label{sec:notation}
If $X$ and $Y$ are Banach spaces and $A:\Dom(A)\subset X\rightarrow Y$ is a linear operator we denote by $\Dom(A)$, $\ker(A)$, and $\ran(A)$ the domain, the kernel, and the range of $A$, respectively. The space of bounded linear operators from $X$ to $Y$ is denoted by $\Lin(X,Y)$ and we write $\Lin(X)$ for $\Lin(X,X)$. If $A: \Dom(A)\subset X\rightarrow X$, then $\gs(A)$
and $\rho(A)$ denote the spectrum
and the \mbox{resolvent} set of $A$, respectively. 
The inner product on a Hilbert space is denoted by $\iprod{\cdot}{\cdot}$.
For $T\in \Lin(X)$ on a Hilbert space $X$ we define $\re T = \frac{1}{2}(T+T^\ast)$, and we write $T\ge cI$ for $c\ge0$ when $T-cI\ge 0$.
The growth bound of a strongly continuous semigroup $T$ is denoted by $\gw_0(T)$. 
We denote $f(x)\lesssim g(x)$ if there exists $M>0$ such that $f(x)\le Mg(x)$ for all parameters $x$.
The truncation of a function $f:I\to X$ defined on an interval $I\supset [0,t]$ to the interval $[0,t]$ is denoted by $\Pt f:[0,t]\to X$.
For $\gb\in\R$ we define $\C_{\gb}^+ = \setm{\gl\in\C}{\re \gl>\gb}$ and 
$H_\infty(\C_{\gb}^+;X)=\setm{f:\C_\gb^+\to X}{f \mbox{ is analytic and } \sup_{\re\gl>\gb}\norm{f(\gl)}<\infty}$.
For a Hilbert space $X$ and for $\gw\in\R$ we define 
\ieq{
\Lp[2]_\gw (0,\infty;X) = \setm{f\in \Lploc[2](0,\infty;X)}{t\mapsto e^{-\gw t}f(t)\in \Lp[2](0,\infty;X)}
}
with norm $\norm{f}_{\Lp[2]_\gw(0,\infty)}=\norm{e^{-\gw \cdot}f}_{\Lp[2](0,\infty)}$.
A function $f:X\to Y$ has \emph{linear growth} if there exists $M>0$ such that $\norm{f(x)}\le M(1+\norm{x})$ 
for all $x\in X$.

\section{Boundary Nodes}
\label{sec:BCS}

In this section we recall and introduce selected background results on abstract boundary control systems. We consider a system 
\begin{subequations}
\label{eq:BCS}
\eqn{
\dot x(t) &= \mcA x(t) + B_i(u(t) + \phi_s(x(t)) + \phi_o(y(t))), \qquad x(0)=x_0 \\
\mcB x(t) &=  Q (u(t) + \phi_s(x(t)) + \phi_o(y(t))) \\
y(t) &= \mcC x(t)
}
\end{subequations}
on a Hilbert space $X$, with an input space $U$ and an output space $Y$. 
Throughout the paper we assume that 
$\phi_s: X\to U$ is locally Lipschitz continuous and has linear growth
 and that $\phi_o:Y\to U$ is globally Lipschitz continuous.
In addition, we assume that the operators in~\eqref{eq:BCS} form a \emph{boundary node} defined below.

\begin{definition}\label{def:Bnode}
	The tuple $(\mcB,\mcA,\mcC,Q,B_i)$ is said to be a \emph{boundary node} on the Hilbert spaces $\BNsp{U_b}{U}{X}{Y}$ if the following hold.
	\begin{enumerate}
		\item[(i)] $\AA: \Dom(\AA)\subset X\to X$, $\mcB\in \mathcal{L}(D(\mcA),U_b)$, $\mcC\in \mathcal{L}(D(\mcA),Y)$, $Q\in \Lin(U,U_b)$, and $B_i\in \Lin(U,X)$;
		\item[(ii)] The restriction $A=\mcA\vert_{\ker(\mcB)}$ with domain $\Dom(A)=\ker(\BB)$ generates a strongly continuous semigroup on $X$;
		\item[(iii)]   $\mcB$ has a bounded right-inverse, i.e., there exists 
     $\mcB^r\in\mathcal{L}(U_b,D(\mcA))$
    such that $\mcB\mcB^r=I$.
	\end{enumerate}	
\end{definition}

We define the \emph{transfer functions} $H: \rho(A)\to \mathcal L(U,X)$ and $P: \rho(A)\to \mathcal L(U,Y)$ of the boundary node so that for $\gl\in \rho(A)$ and $u\in U $ we have $H(\gl)u := x$ and $P(\gl)u := \mcC x$ when $x\in \Dom(\mcA) $ is the unique element such that $(\gl - \mcA)x=B_i u$ and $\mcB x= Qu$. The existence and analyticity of the transfer functions follow from~\citel{NicPau25}{Prop. 2.4}.
We also note that $H(\gl)=H_b(\gl)Q + (\gl-A)\inv B_i$ and $P(\gl)=P_b(\gl)Q + \mcC(\gl-A)\inv B_i$, where $H_b$ and $P_b$ are the transfer functions of the boundary node $(\mcB,\mcA,\mcC,I,0)$ on $\BNsp{U_b}{U_b}{X}{Y}$.
We have $H(\gl)\in \Lin(U,\Dom(\mcA))$, $\mcB H(\gl)=Q$, $P(\gl)=\mcC H(\gl)$ and $(\gl-\mcA)H(\gl)=B_i$ for $\gl\in\rho(A)$ by~\citel{NicPau25}{Prop.~2.4}.

\begin{definition}
\label{def:Solutionsnonlin}
Let $(\mcB,\mcA,\mcC,Q,B_i)$ be a boundary node on $\BNsp{U_b}{U}{X}{Y}$.
A triple $(x,u,y)$ is called a \emph{classical solution of} \eqref{eq:BCS} (\emph{on} $\zinf$) if
\begin{itemize}
    \item $x\in C^1(\zinf;X)$, $u\in C(\zinf;U)$, and $y\in C(\zinf;Y)$
    \item $x(t)\in \Dom(\mcA)$ and 
     \eqref{eq:BCS} hold for all $t\geq 0$.
\end{itemize}
A triple $(x,u,y)$ is called a \emph{generalised solution of}~\eqref{eq:BCS} (\emph{on} $\zinf$) if 
\begin{itemize}
    \item $x\in C(\zinf;X)$, 
    $u\in\Lploc[2](0,\infty;U)$, and $y\in\Lploc[2](0,\infty;Y)$
    \item there exist classical solutions $(x_k,u_k,y_k)$, $k\in\N$,  of  \eqref{eq:BCS} on $\zinf$ such that
$(\Pt[\tau] x_k,\Pt[\tau] u_k,\Pt[\tau] y_k)\tp\to (\Pt[\tau] x,\Pt[\tau] u,\Pt[\tau] y)\tp$  for every $\tau>0$ 
as $k\to \infty$ in $C([0,\tau];X)\times L^2(0,\tau;U) \times L^2(0,\tau;Y)$.
\end{itemize}
\end{definition}

\subsection{Well-Posedness of Boundary Nodes}

In this section we define well-posedness of a boundary node and analyse the existence of solutions of~\eqref{eq:BCS}.
When $(\mcB,\mcA,\mcC,Q,B_i)$ is a boundary node on $\BNsp{U_b}{U}{X}{Y}$, we have from~\citel{MalSta06}{Lem. 2.6} that for every
 $x_0\in \Dom(\mcA)$ and $u\in C^2([0,\tau];U)$ satisfying $\mcB x_0=Qu(0)$ 
the equation~\eqref{eq:BCS} has a unique classical solution $(x,u,y)$ satisfying $x(0)=x_0$.
We use these solutions to define the \emph{input map} $\Phi_\tau : C_\ell^2([0,\tau];U)\to X$, \emph{output map} $\Psi_\tau: \Dom(\mcA)\to L^2(0,\tau;Y)$, and \emph{input-output map} $\F_\tau: C_\ell^2([0,\tau];U)\to L^2(0,\tau;Y)$, where $\tau>0$ and $C_\ell^2([0,\tau];U):=\setm{f\in C^2([0,\tau];U)}{u(0)=0}$.
More precisely, we define $\Phi_\tau u =  x(\tau)$ and $\F_\tau u=\Pt[\tau] y$, where $(x,u,y)$ is the classical solution of~\eqref{eq:BCS}  satisfying $x(0)=x_0$ corresponding to $u\in C_\ell^2([0,\tau];U)$ and $x_0=0$. Moreover, we define $\Psi_\tau x_0= \Pt[\tau]y$ where $(x,u,y)$ is the classical solution of~\eqref{eq:BCS} satisfying $x(0)=x_0$ with $u\equiv 0$ and $x_0\in \Dom(A)$.
We define the well-posedness of a boundary node following~\citel{JacZwa12book}{Def.~13.1.3} and~\cite{Sta05book}.

\begin{definition}
\label{def:BCSwellposed}
Let $(\mcB ,\mcA ,\mcC ,Q,B_i)$ be a boundary node on $\BNsp{U_b}{U}{X}{Y}$.
\begin{itemize}
\item The boundary node is \emph{(externally) well-posed} if there exist $\tau>0$ and $M_\tau>0$ such that
for all $x_0\in \Dom(\mcA)$ and $u\in C^2([0,\tau];U)$ with $\mcB x_0=Qu(0)$ the classical solution $(x,u,y)$ of~\eqref{eq:BCS} satisfying $x(0)=x_0$ satisfies
\eqn{
\label{eq:BCSWPdef}
\norm{x(\tau)}_X^2 + \int_0^\tau \norm{y(t)}_Y^2dt
\leq M_\tau \left( \norm{x_0}_X^2 + \int_0^\tau \norm{u(t)}_U^2 dt \right);
}
\item
The boundary node has a \emph{well-posed input map} if there exist $\tau>0$ and $M_\tau>0$ such that
$\norm{\Phi_\tau u}\le M_\tau \norm{u}_{L^2(0,\tau)}$ for all $u\in C_\ell^2([0,\tau];U)$.
\item
The boundary node has a \emph{well-posed output map} if there exist $\tau>0$ and $M_\tau>0$ such that $\norm{\Psi_\tau x_0}_{L^2(0,\tau)} \le M_\tau \norm{x_0}_X$ for all $x_0\in \Dom(A)$.
\end{itemize}
\end{definition}

We note that by definition a well-posed boundary node has a well-posed input map and a well-posed output map.
We have from~\cite{MalSta06} and~\cite{Sta05book,TucWei14} that if one of the estimates in \cref{def:BCSwellposed} holds for some $\tau>0$, then it also holds for every $\tau>0$ (with a modified constant $M_\tau>0$).
If a boundary node has a well-posed input map, then $\Phi_\tau$ extend to operators $\Phi_\tau\in \Lin(L^2(0,\tau;U),X)$, and if it has a well-posed output map, then $\Psi_\tau$ extend to operators $\Psi_\tau\in \Lin(X,L^2(0,\tau;Y))$.
Finally, if the boundary node is well-posed, then $\F_\tau$ extend to operators $\F_\tau\in \Lin(L^2(0,\tau;U),L^2(0,\tau;Y))$. 
In the last case the tuple $(T,\Phi,\Psi,\F)$ with $\Phi=(\Phi_\tau)_{\tau\ge 0}$, $\Psi=(\Psi_\tau)_{\tau\ge 0}$, and $\F=(\F_\tau)_{\tau\ge 0}$ is a well-posed linear system in the sense of~\cite{TucWei14}.

\begin{remark}
\label{rem:WPtestclass}
It is possible to test well-posedness with a smaller class of inputs and initial conditions. Indeed, 
$(\mcB ,\mcA ,\mcC ,Q,B_i)$ is well-posed if 
 there exist $\tau>0$ and $M_\tau>0$ such that~\eqref{eq:BCSWPdef} holds
for all
classical solutions $(x,u,y)$ corresponding to $x_0\in \Dom(A)$
 and 
$u\in C_\ell^2([0,\tau];U)$.
\end{remark}

The following two results show that generalised solutions of~\eqref{eq:BCS} coincide with \emph{mild solutions} of the linear (with $\phi_s=0$ and $\phi_o=0$) and the nonlinear system~\eqref{eq:BCS}.

\begin{proposition}[{\citel{FkiPau26arxiv}{Prop.~2.5}}]
\label{prp:WPBCSsolutions}
Let $(\mcB ,\mcA ,\mcC ,Q,B_i)$ be a well-posed boundary node on $\BNsp{U_b}{U}{X}{Y}$ and let $(T,\Phi,\Psi,\F)$ be the associated well-posed linear system. Let $\phi_s=0$ and $\phi_o=0$.
If $x_0\in \Dom(\mcA)$ and $u\in \Hloc{1}(0,\infty;U)$ satisfy $\mcB x_0= Qu(0)$, 
then~\eqref{eq:BCS} has a unique classical solution $(x,u,y)$ on $\zinf$. This solution is defined by
\begin{subequations}
\label{eq:BCSmildsol}
\eqn{
x(t) &= T(t)x_0 + \Phi_t \Pt u,  &&t\ge 0,\\
\Pt y &= \Psi_t x_0 + \F_t \Pt u,  &&t\ge0,
}
\end{subequations}
and we have $y\in \Hloc{1}(0,\infty;Y)$.
If $x_0\in X$ and $u\in \Lploc[2](0,\infty;U)$,  
then~\eqref{eq:BCS} has a unique generalised solution $(x,u,y)$ on $\zinf$ which is defined by~\eqref{eq:BCSmildsol}.
\end{proposition}

\begin{proposition}
\label{prp:BCSNLsolutions}
Let $(\mcB ,\mcA ,\mcC ,Q,B_i)$ be a well-posed boundary node on
$(U_b,U,X,Y)$
 and let $(T,\Phi,\Psi,\F)$ be the associated well-posed linear system.
Assume that 
$\phi_s: X\to U$ is locally Lipschitz continuous and has linear growth
 and that $\phi_o:Y\to U$ is globally Lipschitz continuous.
If $(x,u,y)$ is a generalised solution of~\eqref{eq:BCS},
then 
\begin{subequations}
\label{eq:BCSNLmildsol}
\eqn{
x(t) &= T(t)x(0) + \Phi_t \Pt (u+\phi_s(x)+\phi_o(y)),  &&t\ge 0,\\
\Pt y &= \Psi_t x(0) + \F_t \Pt (u+\phi_s(x)+\phi_o(y)),  &&t\ge0.
}
\end{subequations}
 Conversely, if $x\in C(\zinf;X)$, $u\in \Lploc[2](0,\infty;U)$, and $y\in \Lploc[2](0,\infty;Y)$ satisfy~\eqref{eq:BCSNLmildsol}, then $(x,u,y)$ is a generalised solution of~\eqref{eq:BCS} on $\zinf$.
\end{proposition}

\begin{proof}
Let $(x,u,y)$ be a generalised solution of~\eqref{eq:BCS}.
 Let $(x_k,u_k,y_k)$, $k\in \N$, be classical solutions of~\eqref{eq:BCS} 
with the properties in \cref{def:Solutionsnonlin}.
Our assumptions on $\phi_s$ and $\phi_o$ imply that if we define $u_0:=u+\phi_s(x)+\phi_o(y)$ 
 and $u_{0k}:=u_k+\phi_s(x_k)+\phi_o(y_k)$, $k\in \N$, then $u_0,u_{0k}\in \Lploc[2](0,\infty;U)$ for all $k\in \N$ and
 $\norm{\Pt(u_{0k}-u_0)}_{L^2(0,t)}\to 0$ as $k\to \infty$ for all $t\ge0$.
\cref{prp:WPBCSsolutions} implies that
$(x_k(t), \Pt y_k)\tp = \Sigma_t(x_k(0), \Pt u_{0k})\tp$
for all $t\ge 0$ and $k\in\N$, where
\eq{
\Sigma_t =\pmat{T(t)& \Phi_t\\ \Psi_t& \F_t}, \qquad t\ge 0.
}
For every $t\ge0$
the properties $\norm{x_k(0)- x(0)}\to 0$ and
 $\norm{\Pt(u_{0k}-u_0)}_{L^2(0,t)}\to 0$ as $k\to \infty$ imply that
\eq{
\pmat{x(t)\\ \Pt y} 
= \lim_{k\to \infty} \pmat{x_k(t)\\ \Pt y_k}
=   \Sigma_t \pmat{x(0)\\ \Pt u_0}
=   \Sigma_t \pmat{x(0)\\ \Pt (u+\phi_s(x)+\phi_o(y))},
}
where the limit is taken in $X\times L^2(0,t;Y)$.
Thus~\eqref{eq:BCSNLmildsol} hold.

Now let $x\in C(\zinf;X)$, $u\in \Lploc[2](0,\infty;U)$, and $y\in \Lploc[2](0,\infty;Y)$ be such that~\eqref{eq:BCSNLmildsol} hold.
Our assumptions on $\phi_s$ and $\phi_o$ imply that  $u_0:=u+\phi_s(x)+\phi_o(y)\in \Lploc[2](0,\infty;U)$. Let $(x_{0k})_{k\in\N}\subset \Dom(A)$ and $(u_{0k})_{k\in \N}\subset \Hloc{1}(0,\infty;U)$ be such that $u_{0k}(0)=0$, $k\in\N$, and $x_{0k}\to x(0)$ and $\norm{\Pt (u_{0k}-u)}_{\Lp[2]}\to 0$ as $k\to \infty$ for all $t>0$.
If we define $x_k:\zinf\to X$ and $y_k:\zinf\to Y $ by
$(x_k(t), \Pt y_k)\tp = \Sigma_t(x_{0k}, \Pt u_{0k})\tp$
for all $t\ge 0$ and $k\in\N$, then \cref{prp:WPBCSsolutions} implies that $x_k\in C^1(\zinf;X)$ and $y_k\in \Hloc{1}(0,\infty;Y)$.
 If we define $u_k:= u_{0k}-\phi_s(x_k)-\phi_o(y_k)$, then $u_k\in \Hloc{1}(0,\infty;U)$ due to our assumptions on $\phi_s$ and $\phi_o$, and \cref{prp:WPBCSsolutions} further implies that
 for all $t\ge 0$ we have $x_k(t)\in \Dom(\AA)$ and
\eq{
\dot x_k(t) 
&= \mcA x_k(t) + B_iu_{0k}(t)
 \\
\mcB x_k(t) &=  Q u_{0k}(t) 
\\
y_k(t) &= \mcC x_k(t).
}
Since $u_{0k}(t)= u_k(t)+\phi_s(x_k(t))+\phi_o(y_k(t)) $, $(x_k,u_k,y_k)$ is a classical solution of~\eqref{eq:BCS} for every $k\in \N$.
Moreover, for every $t>0$ 
we have
\eq{
\MoveEqLeft \norm{\Pt (x_k-x)}_{\infty}^2
+ \norm{\Pt (y_k-y)}_{L^2}^2(0,t)\\
&\le \norm{\Sigma_t} \left( \norm{x_{k0}-x(0)}_X^2 + \norm{\Pt (u_{0k}-u_0)}_{L^2(0,t)}^2  \right)\to 0
}
as $k\to \infty$ and
further
\eq{
\MoveEqLeft[1]\norm{\Pt (u_k-u)}_{L^2(0,t)}\\
&= 
\norm{\Pt (u_{0k}-u_0)
+\Pt (\phi_s(x_k)-\phi_s(x))
+\Pt (\phi_o(y_k)-\phi_o(y))}_{L^2(0,t)}\to 0
}
as $k\to \infty$ due to our assumptions on $\phi_s$ and $\phi_o$. This shows that $(x,u,y)$ is a generalised solution of~\eqref{eq:BCS} and completes the proof.
\end{proof}

The next lemma is useful in proving well-posedness of a boundary node.

\begin{lemma}[{\citel{FkiPau26arxiv}{Lem.~2.8}}]
\label{lem:IPBCSWP}
Let 
 $(\mcB,\mcA,\mcC,I,0)$ be a boundary node on the Hilbert spaces $\BNsp{U}{U}{X}{U}$ with transfer function $P$ and
 let $K\in \Lin(U)$ be such that $\re K\ge cI$ for some $c>0$.
If 
$\re \iprod{\mcA x}{x}\le \re \iprod{\mcB x}{\mcC x}$ for $ x\in \Dom(\mcA)$,
then $A$ generates a contraction semigroup on $X$ and the following hold.
\begin{itemize}
\setlength{\itemsep}{1ex}
 \item[\textup{(a)}] 
$(\mcB,\mcA,\mcC,I,0)$ is well-posed if and only if 
 $\sup_{s\in\R}\norm{P(\gw+is)}<\infty$ for some $\gw>0$.
 \item[\textup{(b)}] 
 $(\mcB+K\mcC,\mcA,\mcC,I,0)$ is a well-posed boundary node.
 \item[\textup{(c)}]  $(I+KP(\cdot))\inv\in H_\infty(\C_0^+; \Lin(U))$.
\end{itemize}
\end{lemma}

\subsection{Feedback and Cascades of Boundary Nodes}

In this section we investigate the well-posedness of boundary nodes
under output feedback of the form $u(t)=Ky(t)$ and in linear and nonlinear cascade connections.

\begin{proposition}[{\citel{FkiPau26arxiv}{Prop.~2.9}}]
\label{prp:BCSfeedback}
Assume that $(\mcB,\mcA,\mcC,Q,B_i)$ is a well-posed boundary node on the Hilbert spaces $\BNsp{U_b}{U}{X}{Y}$ with transfer function $P$. If $K\in \Lin(Y,U)$ and $(I-KP(\cdot ))\inv \in H_\infty(\C_{\gb}^+;\Lin(U))$ for some $\gb\in \R$, then $(\mcB-QK\mcC,\mcA+B_iK\mcC,\mcC,Q,B_i)$ is a well-posed boundary node on $(U_b,X,Y,U)$.
\end{proposition}

The following result concerns the well-posedness of a linear cascade connection.

\begin{proposition}[{\citel{FkiPau26arxiv}{Prop.~2.10 \& 2.11}}]
\label{prp:CascWP}
Let $(\mcB_1,\mcA_1,\mcC_1,Q_1, B_{i1})$ be a well-posed boundary node on $\BNsp{U_{b1}}{U_1}{X_1}{Y_1}$
and let $(\mcB_2,\mcA_2,\mcC_2,Q_2, B_{i2})$ be a well-posed boundary node on $\BNsp{U_{b2}}{U_2}{X_2}{Y_2}$.
 If $K\in \Lin(Y_2,U_1)$, then 
\eq{
\left( \pmat{\mcB_1 & -Q_1K\mcC_2 \\ 0 & \mcB_2} 
, \pmat{\mcA_1 & B_{i1} K\mcC_2 \\ 0 & \mcA_2}, \pmat{\mcC_1 & 0 \\ 0 & \mcC_2},  \pmat{Q_1 & 0 \\ 0 & Q_2}, \pmat{B_{i1} & 0 \\ 0 & B_{i2}}\right)
}
is a well-posed boundary node on $\BNsp{U_{b1}\times U_{b2}}{U_1\times U_2}{X_1\times X_2}{Y_1\times Y_2}$.
Moreover, $\gw_0(T)=\max \set{\gw_0(T_1), \gw_0(T_2)}$,
 where
$T_1$ and $T_2$ are the semigroups generated by $\mcA_1\vert_{\ker(\mcB_1)}$ and $\mcA_2\vert_{\ker(\mcB_2)}$, respectively.
\end{proposition}

The following proposition provides conditions for the existence of solutions of a  nonlinear cascade connection of two boundary nodes.

\begin{proposition}
\label{prp:BNcascadeNL}
Let $(\mcB_1,\mcA_1,\mcC_1,Q_1, B_{i1})$ be a well-posed boundary node on $\BNsp{U_{b1}}{U_1}{X_1}{Y_1}$
 and let $(\mcB_2,\mcA_2,\mcC_2,Q_2, B_{i2})$ be a well-posed boundary node on
 $(U_{b2},U_2,X_2,Y_2)$. 
Assume further $\phi_s: X_2\to U_1$ is locally Lipschitz continuous and has linear growth, that $\phi_o:Y_2\to U_1$ is globally Lipschitz continuous, and $K\in \Lin(Y_2,Y_1)$.
Then for every $x_0=(x_{10},x_{20})\tp\in X_1\times X_2$ and $u=(u_1,u_2)\tp\in \Lploc[2](0,\infty;U_1\times U_2)$ the system%
\begin{subequations}%
\label{eq:CascSysNonlin}%
\eqn{%
\ddb{t} \pmat{x_1(t)\\x_2(t)} &=    \pmat{\mcA_1 x_1(t) + B_{i1}(u_1(t)+\phi_s(x_2(t))+\phi_o(y_2(t)))\\\mcA_2 x_2(t)+ B_{i2}u_2(t)}\\
\pmat{\mcB_1 x_1(t)  \\  \mcB_2 x_2(t)}
&= \pmat{Q_1(u_1(t)+\phi_s(x_2(t))+\phi_o(y_2(t)))\\ Q_2u_2(t)}
\\
\pmat{y_1(t)\\ y_2(t)} &= \pmat{\CC_1 x_1(t) +K\CC_2 x_2(t)\\ \CC_2x_2(t)}
}
\end{subequations}
with state $x=(x_1,x_2)\tp$ and output $y=(y_1,y_2)\tp$
has a unique generalised solution $( x,u,y)$ satisfying $x(0)=x_0$.
 If 
$x_0=(x_{10},x_{20})\tp\in \Dom(\AA_1)\times \Dom(\AA_2)$ and $u=(u_1,u_2)\tp\in \Hloc{1}(0,\infty;U_1\times U_2)$ are such that 
$\mcB_1 x_{10}=Q_1(u_1(0)+\phi_s(x_{20})+\phi_o(\CC_2x_{20})) $
and   $\mcB_2 x_{20}=Q_2 u_2(0) $, 
then $( x,u,y)$ is a classical solution of the system.
For every $\tau >0$ and $R>0$ there exists $M_{\tau,R}>0$ such that any pair of generalised solutions $(x,u,y)$ and $(x',u',y')$ of~\eqref{eq:CascSysNonlin} satisfy
\eq{
\norm{\Pt[\tau](x-x')}_\infty + 
\norm{\Pt[\tau](y-y')}_{L^2} \le 
M_{\tau,R}\left(
\norm{x(0)-x'(0)} + 
\norm{\Pt[\tau](u-u')}_{L^2} 
  \right) 
}
whenever 
$\norm{x(0)}\le R$, 
$\norm{x'(0)}\le R$, 
$\norm{\Pt[\tau]u}\le R$, 
and $\norm{\Pt[\tau]u'}\le R$.
If $\phi_s$ is globally Lipschitz, $M_{\tau,R}$ can be chosen independently of $R>0$.
\end{proposition}

\begin{proof}
 Let $(T_1,\Phi^1,\Psi^1,\F^1)$
and 
 $(T_2,\Phi^2,\Psi^2,\F^2)$ be
 the well-posed linear systems associated with the boundary nodes
 $(\BB_1,\AA_1,\CC_1,Q_1,B_{i1})$
and 
 $(\BB_2,\AA_2,$ $\CC_2,Q_2,B_{i2})$, respectively, and denote
\eq{
\Sigma_t^1 =\pmat{T_1(t)& \Phi_t^1\\ \Psi_t^1& \F_t^1}
\qquad \mbox{and} \qquad
\Sigma_t^2 =\pmat{T_2(t)& \Phi_t^2\\ \Psi_t^2& \F_t^2}
}
for $t\ge 0$.
Let 
 $x_0=(x_{10},x_{20})\tp\in X_1\times X_2$ and $u=(u_1,u_2)\tp\in \Lploc[2](0,\infty;U_1\times U_2)$.
Since $(\BB_2,\AA_2,\CC_2,Q_2,B_{i2})$ is well-posed, we have from \cref{prp:WPBCSsolutions} that the system%
\begin{subequations}
\label{eq:CascSysNonlin2nd}
\eqn{
\dot x_2(t) &=    \mcA_2 x_2(t)+ B_{i2}u_2(t)\\
 \mcB_2 x_2(t) &=  Q_2u_2(t)
\\
 y_2(t) &= \CC_2x_2(t)
}
\end{subequations}
 has a unique generalised solution $(x_2,u_2,y_2)$ satisfying $x_2(0)=x_{20}$.
Let $(u_2^k)_{k\in\N}\subset\Hloc{1}(0,\infty;U_2)$ be such that $\Pt[\tau]u_2^k\to \Pt[\tau]u_2$ for all $\tau>0$.
Since $\BB_2\in \Lin(\Dom(\AA_2),U_2)$ is surjective and $\Dom(A_2)=\ker(\BB_2)$ is dense in $X_2$, we can choose a sequence $(x_{20}^k)_{k\in\N}\subset \Dom(\AA_2)$ such that 
$ \BB_2 x_{20}^k=u_2^k(0) $ and $x_{20}^k\to x_{20}$ as $k\to\infty$.
If we define $x_2$ and $y_2$ by $(x_2^k(t),\Pt y_2^k)=\Sigma_t^2(x_{20}^k,\Pt u_2^k )$, $t\ge 0$, then \cref{prp:WPBCSsolutions} implies that
 $(x_2^k,u_2^k,y_2^k)$, $k\in \N$, are classical solutions of~\eqref{eq:CascSysNonlin2nd} and 
 $(\Pt[\tau] x_2^k,\Pt[\tau] u_2^k,\Pt[\tau] y_2^k)\tp\to (\Pt[\tau] x_2,\Pt[\tau] u_2,\Pt[\tau] y_2)\tp$ as $k\to \infty$  in $C([0,\tau];X)\times \Lp[2](0,\tau;U_2)\times \Lp[2](0,\tau;Y_2)$ for every $\tau>0$.
The assumptions on $\phi_s$ and $\phi_o$ imply that
 $t\mapsto  \phi_s(x_2(t))+\phi_o(y_2(t))\in \Lploc[2](0,\infty;U_1)$,
that
 $t\mapsto  \phi_s(x_2^k(t))+\phi_o(y_2^k(t))\in \Hloc{1}(0,\infty;U_1)$, $k\in\N$,
and that
$ \Pt[\tau] (\phi_s(x_2^k(\cdot))+\phi_o(y_2^k(\cdot)))
\to\Pt[\tau] (\phi_s(x_2(\cdot))+\phi_o(y_2(\cdot))) $ as $k\to\infty$ in $\Lp[2](0,\tau;U_1)$ for every $\tau>0$.
Let $(u_1^k)_{k\in\N}\subset\Hloc{1}(0,\infty;U_1)$ be such that $\Pt[\tau]u_1^k\to \Pt[\tau]u_1$ for all $\tau>0$.
Since $\BB_1\in \Lin(\Dom(\AA_1),U_1)$ is surjective and $\Dom(A_1)=\ker(\BB_1)$ is dense in $X_1$, we can choose $(x_{10}^k)_{k\in\N}\subset \Dom(\AA_1)$ such that 
$ \BB_1 x_{10}^k=u_1^k(0)+ \phi_s(x_2^k(0))+\phi_o(y_2^k(0)) $ and $x_{10}^k\to x_{10}$ as $k\to\infty$.
If for $k\in\N$ we define $x_1^k: \zinf \to X_1$ and $y_1^k:\zinf \to Y_1$ by
\eq{
\pmat{ x_1^k(t)\\ \Pt y_1^k }
= \Sigma_t^1  \pmat{ x_{10}^k\\\Pt (u_1^k+ \phi_s(x_2^k)+\phi_o(y_2^k)) }
+ \pmat{0\\K\Pt y_2^k}, \qquad
t\ge0,
}
then \cref{prp:WPBCSsolutions} implies that
$(x^k,u^k,y^k)$ with
$x^k= (x_1^k,x_2^k)\tp$, $ u^k=(u_1^k,u_2^k)\tp$, and $y^k=(y_1^k,y_2^k)\tp$
 are a classical solutions of~\eqref{eq:CascSysNonlin}. 
Finally, if we define 
$x_1\in C(\zinf ; X_1)$ and $y_1\in \Lploc[2] (0,\infty;Y_1)$ by
\eqn{
\label{eq:CascSysNonlinx1y1}
\pmat{ x_1(t)\\ \Pt y_1 }
= \Sigma_t^1  \pmat{ x_{10}\\\Pt (u_1+ \phi_s(x_2)+\phi_o(y_2)) }
+ \pmat{0\\K\Pt y_2}, \qquad
t\ge0,
}
then the property $ \Pt[\tau] (\phi_s(x_2^k)+\phi_o(y_2^k))
\to\Pt[\tau] (\phi_s(x_2)+\phi_o(y_2)) $ as $k\to\infty$ in $\Lp[2](0,\tau;U_1)$ for $\tau>0$
 implies that 
$\Pt[\tau] x_1^k\to \Pt[\tau] x_1$
$C([0,\tau];X)$
 and
$\Pt[\tau] y_1^k\to \Pt[\tau] y_1$
in $\Lp[2](0,\tau;Y_1)$ as $k\to\infty$ for every $\tau>0$. 
Defining $x=(x_1,x_2)$ and $y=(y_1,y_2)$ our analysis shows that
 $(\Pt[\tau]x^k,\Pt[\tau] u^k,\Pt[\tau]y^k)\tp\to (\Pt[\tau]x,\Pt[\tau] u,\Pt[\tau]y)\tp $ in $C([0,\tau];X)\times \Lp[2](0,\tau;U)\times \Lp[2](0,\tau;Y)$ as $k\to\infty$ for all $\tau>0$, and thus $(x,u,y)$ is a generalised solution of~\eqref{eq:CascSysNonlin}.

Assume now that 
$x_0=(x_{10},x_{20})\tp\in \Dom(\AA_1)\times \Dom(\AA_2)$ and $u=(u_1,u_2)\tp\in \Hloc{1}(0,\infty;U_1\times U_2)$ satisfy
$\mcB_1 x_{10}=Q_1(u_1(0)+\phi_s(x_{20})+\phi_o(\CC_2x_{20})) $
and   $\mcB_2 x_{20}=Q_2 u_2(0) $. 
Then the generalised solution $(x_2,u_2,y_2)$ of~\eqref{eq:CascSysNonlin2nd} above is in fact a classical solution by \cref{prp:WPBCSsolutions}, and $y_2\in \Hloc{1}(0,\infty;U_2)$.
Our assumptions on $\phi_s$ and $\phi_o$ imply that
$ u_1+  \phi_s(x_2)+\phi_o(y_2)\in \Hloc{1}(0,\infty;U_1)$. 
Since 
$\mcB_1 x_{10}=Q_1(u_1(0)+\phi_s(x_{20})+\phi_o(\CC_2x_{20})) $,  \cref{prp:WPBCSsolutions} and the structure of the system~\eqref{eq:CascSysNonlin} imply that if we define $x_1\in C(\zinf;X_1)$ and $y_1\in \Lploc[2](0,\infty;Y_1)$ by~\eqref{eq:CascSysNonlinx1y1}, then $(x,u,y)$ is a classical solution of~\eqref{eq:CascSysNonlin} and $x(0)=x_0$.

Finally, we will prove the estimate for a pair $(x,u,y)$ and $(x',u',y')$ of generalised solutions. This estimate in particular implies the uniqueness of generalised solutions.
To this end, denote the Lipschitz constant of $\phi_o$ by $L_o>0$ and fix $\tau>0$ and $R>0$. 
Let 
$x_0=(x_{10},x_{20})\tp\in \Dom(\AA_1)\times \Dom(\AA_2)$
$x_0'=(x_{10}',x_{20}')\tp\in \Dom(\AA_1)\times \Dom(\AA_2)$
 and $u=(u_1,u_2)\tp\in \Hloc{1}(0,\infty;U_1\times U_2)$
 and $u'=(u_1',u_2')\tp\in \Hloc{1}(0,\infty;U_1\times U_2)$
be such that
$\norm{x_0}\le R$, 
$\norm{x'_0}\le R$, 
$\norm{\Pt[\tau]u}_{L^2}\le R$, 
and $\norm{\Pt[\tau]u'}_{L^2}\le R$
and such that
$\mcB_1 x_{10}=Q_1(u_1(0)+\phi_s(x_{20})+\phi_o(\CC_2x_{20})) $,
   $\mcB_2 x_{20}=Q_2 u_2(0) $,
$\mcB_1 x_{10}'=Q_1(u_1'(0)+\phi_s(x_{20}')+\phi_o(\CC_2x_{20}')) $,
and   $\mcB_2 x_{20}'=Q_2 u_2'(0) $.
Let $(x,u,y)$ classical solutions of~\eqref{eq:CascSysNonlin} satisfying $x(0)=x_0$ and
let $(x',u',y')$ be the classical solutions of~\eqref{eq:CascSysNonlin} satisfying $x'(0)=x_0'$.
 The structure of~\eqref{eq:CascSysNonlin} and \cref{prp:WPBCSsolutions} imply that 
\eq{
x_2(t) &= T_2(t)x_2(0) + \Phi_t^2 \Pt u_2,  \qquad  t\in [0,\tau]\\
\Pt[\tau] y_2 &= \Psi_\tau^2 x_2(0) + \F_\tau^2 \Pt[\tau] u_2
}
and similarly for $(x',u',y')$. In particular, there exists a constant $M_2>0$ depending only on 
 $(T_2,\Phi^2,\Psi^2,\F^2)$ and $\tau>0$
such that
\eq{
\norm{\Pt[\tau]x_2}_\infty+ \norm{\Pt[\tau]y_2}_{L^2} &\le M_2 \left( \norm{x_2(0)} + \norm{\Pt[\tau]u_2} \right)\le 2M_2 R\\
\norm{\Pt[\tau]x_2'}_\infty+ \norm{\Pt[\tau]y_2'}_{L^2} &\le M_2 \left( \norm{x_2'(0)} + \norm{\Pt[\tau]u_2'} \right)\le 2M_2 R.
}
Moreover, if we define $z=(z_1,z_2)\tp:=x-x'$, $v=(v_1,v_2)\tp:=u-u'$, and $w=(w_1,w_2)\tp:=y-y'$, then 
\eqn{
\label{eq:CascSysNonlinz2est}
\norm{\Pt[\tau]z_2}_\infty+ \norm{\Pt[\tau]w_2}_{L^2} &\le M_2 \left( \norm{z_2(0)} + \norm{\Pt[\tau]v_2} \right).
}
Our assumptions on $\phi_s$ and $\phi_o$ imply that $\psi:= v_1+\phi_s(x_2)-\phi_s(x_2')+\phi_o(y_2)-\phi_o(y_2')$ satisfies $\psi\in \Hloc{1}(0,\infty;U_1)$. 
Since $\phi_s$ is locally Lipschitz and 
$\norm{\Pt[\tau]x_2}_\infty\le 2M_2R $
and
$\norm{\Pt[\tau]x_2'}_\infty\le 2M_2R $, there exists  $L_R>0$ depending only on $M_2$ and $R$  such that 
\eq{
\norm{\Pt[\tau]\psi}_{L^2}
&\le \norm{\Pt[\tau]v_1}_{L^2} + L_R \norm{\Pt[\tau] z_2}_{L^2} + L_o \norm{\Pt[\tau] w_2}_{L^2}\\
&\le \norm{\Pt[\tau]v_1}_{L^2} + (\sqrt{\tau} L_R  + L_o)M_2 \left( \norm{z_2(0)} + \norm{\Pt[\tau]w_2}_{L^2} \right).
}
The structure of~\eqref{eq:CascSysNonlin} and \cref{prp:WPBCSsolutions} imply that 
\eq{
z_1(t) &= T_1(t)z_1(0) + \Phi_t^1 \Pt \psi,  \qquad \qquad t\in [0,\tau]\\
\Pt[\tau] w_1 &= \Psi_\tau^1 z_1(0) + \F_\tau^1 \Pt[\tau] \psi + K\Pt[\tau] w_2.
}
Thus there exists $M_1>0$ depending only on 
 $(T_1,\Phi^1,\Psi^1,\F^1)$ and $\tau>0$ such that 
\eq{
\MoveEqLeft[5]\norm{\Pt[\tau]z_1}_\infty+ \norm{\Pt[\tau]w_1}_{L^2} 
\le M_1 \left( \norm{z_1(0)} + \norm{\Pt[\tau]\psi}_{L^2} \right) + \norm{K}\norm{\Pt[\tau] w_2}_{L^2}\\
&\le M_1 \left( \norm{z_1(0)}  + \norm{\Pt[\tau]v_1}_{L^2} \right)\\
 &\quad+(M_1\sqrt{\tau} L_R  + M_1L_o + \norm{K})M_2 \left( \norm{z_2(0)} + \norm{\Pt[\tau]w_2}_{L^2}\right).
}
Combining with~\eqref{eq:CascSysNonlinz2est} this shows that the estimate in the claim for classical solutions $(x,u,y)$ and $(x',u',y')$ with $u,u'\in \Hloc{1}(0,\infty;U_1\times U_2)$. 
It is straightforward to show that the existence of $M_{\tau,R}>0$ such that  the estimate holds also extends to generalised solutions. 
Finally, if $\phi_s$ is globally Lipschitz, the constant $L_R>0$ above can be chosen to be independently of $R$, which implies the same for $M_{\tau,R}>0$.
\end{proof}

\subsection{Stability}

The following result collects properties of the solutions of~\eqref{eq:BCS} in the case where the associated semigroup is exponentially stable. The notation $L_\gw^2(0,\infty;U)$ is defined in \cref{sec:notation}.

\begin{proposition}
\label{prp:BNstability}
Assume that $(\BB,\AA,\CC,Q,B_i)$ is a well-posed boundary node on \BNspStandard\ and that the semigroup $T$ generated by $A=\AA\vert_{\ker(\BB	)}$ is exponentially stable.
 Let $\phi_s=0$ and $\phi_o=0$. 
For every $\gw_s,\gw>\gw_0(T)$ there exists $M>0$ such that
for $x_0\in X$ and
$u\in \Lploc[2](0,\infty;U)$
 the generalised solution $(x,u,y)$ of~\eqref{eq:BCS} 
satisfies
\eq{
\norm{x(t)} &\le Me^{\gw_s t} \norm{x_0} + Me^{\gw t}\norm{u}_{L_\gw^2(0,\infty)}, \quad  && t\ge 0\\
\norm{x(t)} &\le Me^{\gw_s t} \norm{x_0} + M\norm{u}_{L^\infty(0,\infty)}, \quad && t\ge 0\\
\norm{y}_{L_\gw^2(0,\infty)}
 &\le M \norm{x_0} + M\norm{u}_{L_\gw^2(0,\infty)}.
}
In particular, the following hold:
\begin{itemize}
\setlength{\itemsep}{.8ex}
\item[\textup{(a)}] 
If $u\in L_\gw^2(0,\infty;U)$, then $\sup_{t\ge 0}e^{-\gw t}\norm{x(t)}<\infty$ and 
$y\in L_\gw^2(0,\infty;Y)$.
\item[\textup{(b)}]
If $\gw_0(T)<0$ and $u\in L^\infty(0,\infty;U)$, then $\sup_{t\ge 0}\norm{x(t)}<\infty$.
\end{itemize}
\end{proposition}

\begin{proof}
Let $\Sigma=(T,\Phi,\Psi,\F)$ be the well-posed system associated with the boundary node $(\BB,\AA,\CC,Q,B_i)$. If we denote by $\Psi_\infty$ and $\F_\infty$ the extended output map and extended input-output map~\citel{TucWei14}{Sec.~3}, respectively, of $\Sigma$, then~\citel{Sta05book}{Thm.~2.5.4 \& 2.8.1} and~\citel{Sta05book}{Ex.~2.5.3} imply that 
\ieq{
\norm{\Phi_t \Pt u}_X 
\lesssim e^{\gw t} \norm{u}_{L_\gw^2(0,\infty;U)},
}
$t\ge 0$,
\ieq{
\norm{\Psi_\infty x_0}_{L_\gw^2(0,\infty;Y)}
 \lesssim \norm{x_0},
}
and 
\ieq{
\norm{\F_\infty u}_{L_\gw^2(0,\infty;Y)}
\lesssim \norm{u}_{L_\gw^2(0,\infty;U)}.
}
These together with \cref{prp:WPBCSsolutions} and $\sup_{t\ge 0}e^{-\gw_s t}\norm{T(t)}<\infty$ imply the first and third estimate in the claim.
If $u\in L^\infty(0,\infty;U)$ and $t>0$, applying the first of the above estimates with $\gw\in (\gw_0(T),0)$ to $v\in L_\gw^2(0,\infty;U)$ defined so that $v(s)=u(s)$ for a.e. $s\in [0,t]$ and $v(s)=0$ for $s>t$ shows that 
\ieq{
\norm{\Phi_t \Pt u}_X 
\lesssim e^{\gw t} \norm{v}_{L_\gw^2(0,\infty;U)}
= e^{\gw t} \norm{e^{-\gw \cdot}u}_{L^2(0,t)} \lesssim \norm{u}_{L^\infty} .
}
This implies the second estimate in the claim.
\end{proof}

\subsection{System Transformations}

The following lemma provides a sufficient condition for the possibility to exchange the roles of some of the inputs and outputs of a boundary node. Such partial \emph{flow inversion}~\citel{Sta05book}{Sec.~6.3} is used in the construction of the active disturbance rejection controller in \cref{sec:AbstractController}.

\begin{lemma}[{\citel{NicPau25}{Prop.~2.8}}]
\label{lem:IOswapIP}
Assume that $(\BB,\AA,\CC,I,0)$ is a boundary node on the Hilbert spaces $\BNsp{U}{U}{X}{U}$ satisfying $\re \iprod{\AA x}{x}_X\le \re \iprod{\BB x}{\CC x}_U$  for $x\in \Dom(\AA)$.
Assume that $U=U_1\times U_2$ and denote
\eq{
\BB = \pmat{\BB_1\\ \BB_2}, \quad \CC = \pmat{\CC_1\\\CC_2}, \quad
\BB' = \pmat{\CC_1\\ \BB_2}, \quad 
\mbox{and} \quad
\CC' = \pmat{\BB_1\\\CC_2}.
}
If the transfer function $P_1$ of $(\BB_1,\AA,\CC_1,I,0)$ is such that $P_1(\gl)$ is boundedly invertible for some $\gl\in \rho(A)\cap\overline{\C_0^+}$, then $(\BB',\AA,\CC',I,0)$ is a boundary node on  $\BNsp{U}{U}{X}{U}$.
\end{lemma}

The following lemma presents conditions for the boundary node property and well-posedness under state transformations and changes in the inputs and outputs of the boundary control system.

\begin{lemma}
\label{lem:BnodeModifications}
Let $(\BB,\AA,\CC,Q,B_i)$ be a boundary node on \BNspStandard\ and let $X_0 $, $U_0$, $U_{b0}$, and $Y_0$ be Hilbert spaces. 
Let $S\in \Lin(X,X_0)$ and $S_b\in \Lin(U_b,U_{b0})$ be boundedly invertible and let
$S_i\in \Lin(U_0,U)$, $S_o\in \Lin(Y,Y_0)$, and $D\in \Lin(U_b,Y)$.
 Then the following hold.
\begin{itemize}
\item[\textup{(a)}] 
$(S_b\BB S\inv,S\AA S\inv,\CC S\inv,S_bQ,SB_i)$ is a boundary node on 
$(U_{b0},U,X_0,$ $Y)$
 and $(S_b\BB S\inv,S\AA S\inv,\CC S\inv,S_bQ,SB_i)$ is well-posed if and only if 
 $(\BB,\AA,\CC,Q,B_i)$ is well-posed.
\item[\textup{(b)}] 
 $(\BB,\AA,S_o\CC,QS_i,B_iS_i)$ is a boundary node on $\BNsp{U_b}{U_0}{X}{Y_0}$ and well-posed\-ness of $(\BB,\AA,\CC,Q,B_i)$ implies the well-posedness of $(\BB,\AA,S_o\CC,$ $QS_i,B_iS_i)$.  
\item[\textup{(c)}]
 $(\BB,\AA,\CC+D\BB,Q,B_i)$ is a boundary node on \BNspStandard,
 and well-posedness of $(\BB,\AA,\CC,Q,B_i)$ implies the well-posedness of $(\BB,\AA,\CC+D\BB,Q,B_i)$.  
\end{itemize}
\end{lemma}

\begin{proof}
It is straightforward to verify (a) and (b).
Clearly $(\BB,\AA,\CC+D\BB,Q,B_i)$ is a boundary node. 
Moreover, if $(\BB,\AA,\CC,Q,B_i)$ is well-posed and $(T,\Phi,\Psi,\F)$ is the associated well-posed linear system, then the input maps $\Phi_\tau'$, output maps $\Psi_\tau'$, and input-output maps $\F_\tau'$ of $(\BB,\AA,\CC+D\BB,Q,B_i)$ satisfy $\Phi_\tau' \Pt[\tau]u=\Phi_\tau \Pt[\tau]u$ and $\F_\tau' \Pt[\tau]u=\F_\tau\Pt[\tau]u+DQ\Pt[\tau]u$ for $u\in C_\ell^2([0,\tau];U)$ and $\Psi_\tau' x_0=\Psi_\tau x_0$ for $x_0\in \Dom(A)=\ker(\BB)$. 
These formulas imply that $(\BB,\AA,\CC+D\BB,Q,B_i)$ is well-posed.
\end{proof}

Finally, the following result describes the solutions of the nonlinear system~\eqref{eq:BCS} under transformations of the state and output.

\begin{lemma}
\label{lem:BCSnonlinSolEquivalence}
Let $(\BB,\AA,\CC,Q,B_i)$ be a well-posed boundary node on 
$(U_b,U,$ $X,Y)$.
 Assume that $\phi_s:X\to U$ is locally Lipschitz and has linear growth and $\phi_o:Y\to U$ is globally Lipschitz. Let $X_0 $ and $Y_0$ be Hilbert spaces and let $S\in \Lin(X,X_0)$ be boundedly invertible and $S_o\in \Lin(Y,Y_0)$. Let $\phi_o':Y_0\to U$ be a globally Lipschitz function satisfying $\phi_o'(S_oy)=\phi_o(y)$, $y\in Y$.
Then $(x,u,y)$ is a generalised (resp. classical) solution of the system~\eqref{eq:BCS} 
 if and only if $(z,u,\tilde y)$ with $z=Sx$ and $\tilde y=S_o y$ is a generalised (resp. classical) solution of the system%
\begin{subequations}%
\label{eq:BCSnonlinTrans}%
\eqn{
\dot z(t) &= \mcA' z(t) + B_i'(u(t) + \phi_s'(z(t)) + \phi_o'(\tilde y(t)))\\
\mcB' z(t) &=  Q (u(t) + \phi_s'(z(t)) + \phi_o'(\tilde y(t))) \\
\tilde y(t) &= S_o\mcC' z(t),
}
\end{subequations}
where $\AA'=S\AA S\inv $ with $\Dom(\AA')=S(\Dom(\AA))$, $\BB'=\BB S\inv $, $\CC'=\CC S\inv$, $B_i'=SB_i$, and $\phi_s'=\phi_s(S\inv \cdot)$.
Moreover, the generalised solutions $(x,u,y)$ of~\eqref{eq:BCS} are uniquely determined by $x(0)$ and $u\in \Lploc[2](0,\infty;U)$ if and only if the generalised solutions $(x,u,\tilde y)$ of~\eqref{eq:BCSnonlinTrans} are uniquely determined by $z(0)$ and $u\in \Lploc[2](0,\infty;U)$.
\end{lemma}

\begin{proof}
If $(x,u,y)$ is a classical solution of~\eqref{eq:BCS}, it is immediate that $(z,u,\tilde y)$ with $z=Sx$ and $\tilde y=S_oy$ is a classical solution of~\eqref{eq:BCSnonlinTrans}. Conversely, let $(z,u,\tilde y)$ be a classical solution of~\eqref{eq:BCSnonlinTrans}. Since $ u + \phi_s'(z)+\phi_o'(\tilde y)\in C(\zinf;U)$, we have that $\AA'z(\cdot)\in C(\zinf;X_0)$, and thus $x:=S\inv z(\cdot)\in C(\zinf;\Dom(\AA))$. We can therefore define $y=\CC x(\cdot)\in C(\zinf;Y)$ and the properties of $\phi_o'$ imply that $(x,u,y)$ is a classical solution of~\eqref{eq:BCS}.

To investigate the generalised solutions, let $(T,\Phi,\Psi,\F)$ be the well-posed linear system associated with $(\BB,\AA,\CC,Q,B_i)$. \cref{lem:BnodeModifications} implies that $(\BB',\AA',S_o\CC',Q,B_i')$ is a well-posed boundary node, and clearly its associated well-posed system is $(T', \Phi',\Psi',\F')$, where $T'(t)=ST(t)S\inv$, $\Phi_t' =S\Phi_t$, $\Psi_t' = S_o \Psi_t S\inv$, and $\F'_t =S_o\F_t$ for all $t\ge0$.
Let $(x,u,y)$ be a generalised solution of~\eqref{eq:BCS}. 
Using \cref{prp:BCSNLsolutions} twice shows that 
if we define
$z=Sx$ and
 $\tilde y=S_o y$,
 then  $(z,u,\tilde y)$ is a generalised solution of~\eqref{eq:BCSnonlinTrans}.
Now let $(z,u,\tilde y)$ be a generalised solution of~\eqref{eq:BCSnonlinTrans}.
\cref{prp:BCSNLsolutions} and our assumptions on $\phi_s$ and $\phi_o$ imply that
if we define $x=S\inv z$ and
 $ y=\Psi_\infty x(0)+\F_\infty (u + \phi_s(x)+\phi_o'(\tilde y))\in \Lploc[2](0,\infty;Y)$, then
$(x,u,y)$ is a generalised solution of~\eqref{eq:BCS}.

To prove the claim regarding uniqueness, assume first that the generalised solutions of~\eqref{eq:BCS} 
 are uniquely determined by $x(0)$ and $u\in \Lploc[2](0,\infty;U)$.
If $(z_1,u,\tilde y_1)$ and $(z_2,u,\tilde y_2)$ are two generalised solutions of~\eqref{eq:BCSnonlinTrans} satisfying $z_1(0)=z_2(0)$, then our analysis above shows that there exist $y_1,y_2\in \Lploc[2](0,\infty;Y)$ such that $(x_1,u, y_1)$ and $(x_2,u, y_2)$ are generalised solutions of~\eqref{eq:BCS} and $z_k=Sx_k$ and $\tilde y_k=S_oy_k$ for $k=1,2$. The uniqueness of solutions of~\eqref{eq:BCS} implies that $x_1=x_2$ and $\tilde y_1=\tilde y_2$.
Conversely, assume that
the generalised solutions of~\eqref{eq:BCSnonlinTrans} 
 are uniquely determined by $z(0)$ and $u\in \Lploc[2](0,\infty;U)$.
Let $(x_1,u, y_1)$ and $(x_2,u, y_2)$ be two generalised solutions of~\eqref{eq:BCS} satisfying $x_1(0)=x_2(0)$.
Then $(z_1,u,\tilde y_1)$ and $(z_2,u,\tilde y_2)$ with $z_k=Sx_k$ and $\tilde y_k=S_oy_k$ for $k=1,2$ are generalised solutions of~\eqref{eq:BCSnonlinTrans}, and the uniqueness of solutions implies that $x_1=x_2$ and $S_oy_1=S_oy_2$.
Since 
$ u_k + \phi_s(x_k)+\phi_o(y_k)\in \Lploc[2](0,\infty;U)$ for $k=1,2$,
and since $\phi_o(y_1)=\phi_o'(S_oy_1)=\phi_o'(S_oy_2)=\phi_o(y_2)$,
 \cref{prp:BCSNLsolutions} implies that
\ieq{
y_1-y_2=\Psi_\infty (x_1(0)-x_2(0))+\F_\infty (\phi_s(x_1)-\phi_s(x_2)+\phi_o(y_1)-\phi_o(y_2)) = 0.
}
This completes the proof.
\end{proof}

\section{Stabilising Controller Design}
\label{sec:AbstractController}

In this section we design a stabilising controller for the system
\begin{subequations}
\label{eq:plant}
\eqn{
\dot x(t) &= \mcA x(t) , \hspace{4.8cm} x(0) = x_0\\
\label{eq:plantBC}
\mcB x(t) &= Q(u(t) + d(t)+\phi_s(x(t))+\phi_o(y(t))) \\
y(t) &= \mcC x(t),
}
\end{subequations}
where $(\mcB,\mcA,\mcC,Q,B_i)$ is a boundary node on the Hilbert spaces $\BNsp{U_b}{U}{X}{Y}$. The input of~\eqref{eq:plant} is corrupted by an unknown external
 disturbance signal $d\in \Lploc[2](0,\infty;U)$ as well as nonlinear dependence on the state and output described by the  unknown functions $\phi_s: X\to U$ and $\phi_o:Y\to U$.
We define the \emph{total disturbance} affecting the input as $\dtot = d+\phi_s(x)+\phi_o(y)$. 
We assume that the boundary conditions~\eqref{eq:plantBC} can be decomposed into homogeneous and controlled part so that
$U_b=U_h\times U$, and 
correspondingly the output can be decomposed so that
$Y=Y_d\times Y_m$ and $y(t)=(y_d(t),y_m(t))\tp\in Y_d\times Y_m$, where
$y_d$ are boundary measurements associated to the controlled part of the boundary and $y_m$ contains the rest of the measured outputs. 
More precisely, we assume that 
 $\BB\in \Lin(\Dom(\AA),U_h\times U)$, $Q\in \Lin(U,U_h\times U)$, and $\CC\in \Lin(\Dom(\AA),Y_d\times Y_m)$ have the forms 
\eqn{
\label{eq:BQCstruct}
\BB=\pmat{\BB_h\\ \BB_d} ,
\qquad
 Q=\pmat{0\\I},
 \quad \mbox{and} \quad
\CC = \pmat{\CC_d\\ \CC_m}.
}
 In this decomposition $\BB_h$ represents those static boundary conditions which can be used for output injections.
The general criterion for the choice of $Y_d$ is that the boundary conditions determined by $\BB_d x(t)=u(t)$ should be replaceable with boundary conditions of the form $\CC_d x(t)=\tilde u(t)$ in the dynamics of~\eqref{eq:plant}. This condition is made more precise in Assumption~\ref{ass:ADRCass}.

The active disturbance rejection controller that we introduce has state
 $x_c(t)=(x_s(t),x_i(t),\hat x(t))\tp$ 
on $X_c=X\times X\times X$ and
  has the general form
\begin{subequations}
\label{eq:ADRController}
\eqn{
\dot{x}_s(t)&= \AA x_s(t) + L_i(\CC x_s(t)-y(t))\\
\dot{x}_i(t)&= (\AA+L_i \CC) x_i(t) \\
\dot{\hat{x}}(t)&= \AA\hat{x}(t) + L_i(\CC\hat{x}(t)-y(t)+u_2(t))\\
\label{eq:ADRCcontrxsBC}
(\BB-L \CC ) x_s(t)&=Q u(t)  -Ly(t)\\
\CC_d x_i(t) &=  y_d(t)-\CC_d x_s(t)\\
(\BB_h-L_h \CC)x_i(t)&=  0\\
\label{eq:ADRCcontrxhatBC}
(\BB- L \CC)\hat{x}(t)&= Q(u(t) + \hat d(t)) +L(u_2(t)-y(t))\\
\hat{d}(t)&= (\BB_d-L_d \CC)x_i(t)\\
u(t) &= \KK \hat x(t) - \hat d(t) + u_1(t)
}
\end{subequations}
with initial states $x_s(0)=x_{s0}$, $x_i(0)=x_{i0}$, and $\hat x(0)=\hat x_0$.
Here $\KK$ is a state feedback gain operator and $L_i$ and $L$ are output injection gain operators.
This controller structure follows the design process of the active disturbance rejection control for boundary controlled PDEs~\cite{FenGuo17b,ZhoWei18} and it
 is illustrated in \cref{fig:ContrSchematic3}. The first part of the controller is the \emph{total disturbance estimator} (TDE) which constructs the asymptotic estimate $\hat d$ of $\dtot = d+\phi_s(x) + \phi_o(y)$. This part of the controller itself consists of the \emph{separator} (with state $x_s$), which prestabilises and removes $u$ from the original linear system before the inversion, and the \emph{inverter} (with state $x_i$) which inverts the input and the output of the system to estimate the unknown part $\dtot$ of the input. 
The \emph{observer} (with state $\hat x$) is a \emph{linear} Luenberger-type observer for the linear part of~\eqref{eq:plant}.
The estimate $\hat d $ is used by the observer to both approximate the unknown part $\dtot$ of the input of~\eqref{eq:plant} and to asymptotically cancel the effect of $\dtot$ at the input of the system.
We note that the controller~\eqref{eq:ADRController} is not always associated with a boundary node and it does not necessarily have well-defined dynamics. Instead,~\eqref{eq:ADRController} should be considered as a \emph{controller with an internal loop} in the sense of~\cite{WeiCur97}. 

\begin{figure}[h!]
\begin{center}
\includegraphics[width=0.7\linewidth]{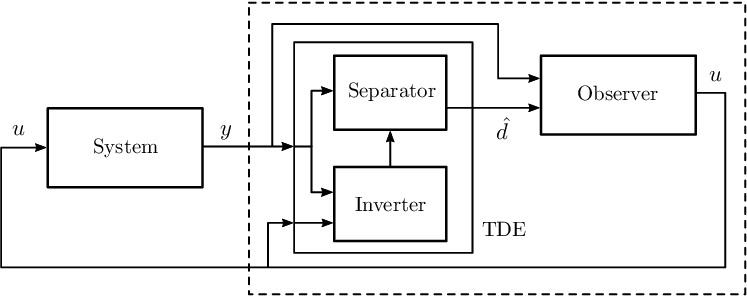}
\caption{The structure of the active disturbance rejection controller.}
\label{fig:ContrSchematic3}
\end{center}
\end{figure}

The closed-loop system consisting of~\eqref{eq:plant} and the controller~\eqref{eq:ADRController} has state 
 $x_e(t)=(x(t),x_s(t),x_i(t),\hat x(t))\tp\in X_e:=X^4$, input  $ u_e(t)= (u_1(t),u_2(t),$ $d(t))\tp \in U_e:=U\times Y\times U$, and output $ y_e(t)= (y(t),\hat y(t),\hat d(t), u_0(t))\tp \in Y_e:=Y\times Y \times U\times U$, where $u_0(t)=\KK \hat x(t)$.  This closed-loop system has the form
\begin{subequations}
\label{eq:ObsCLsysExtIO}
\eqn{
\dot x_e(t) &= \mcA_e x_e(t) + B_{ei}  u_e(t), \hspace{3cm} x_e(0)= x_{e0}\\
\mcB x_e(t) &= Q_e  (u_e(t) + \phi_{es}(x_e(t))+ \phi_{eo}(y_e(t)))\\
 y_e(t) &= \mcC_e x_e(t)
}
\end{subequations}
with $x_{e0}=(x_0,x_{s0},x_{i0},\hat x_0)\tp$.
The boundary conditions~\eqref{eq:plantBC},~\eqref{eq:ADRCcontrxsBC}, and~\eqref{eq:ADRCcontrxhatBC} 
satisfy
\eq{
&\begin{cases}
\mcB x(t) = Q(u(t) + \dtot(t)) \\
(\BB-L \CC ) x_s(t)=Q u(t)  -Ly(t)\\
(\BB- L \CC)\hat{x}(t)= Q(u(t) + \hat d(t)) +L(u_2(t)-y(t))
\end{cases}\\[1ex]
\Leftrightarrow \quad &
\begin{cases}
 (\BB-L \CC ) x(t) -(\BB-L \CC ) x_s(t)  = Q \dtot(t) \\
(\BB-L \CC ) x_s(t)=Q u(t)  -Ly(t)\\
(\BB-L \CC ) x_s(t)-(\BB- L \CC)\hat{x}(t)= -Q \hat d(t) -Lu_2(t).
\end{cases}
}
Thus the operators 
$\mcA_e: \Dom(\mcA_e)\subset X_e\to X_e$,
 $\mcB_e\in \Lin( \Dom(\mcA_e), U_{be})$,
$\mcC_e\in \Lin( \Dom(\mcA_e), Y_e)$, $B_{ei}\in \Lin(U_e, X_e)$, and $Q_e\in \Lin(U_e, U_{be})$ 
in~\eqref{eq:ObsCLsysExtIO} can be defined so that
 $\Dom(\mcA_e)= \setm{(x,x_s,x_i,\hat x)\tp\in\Dom(\mcA)^4}{\CC_d(x-x_s-x_i)=0 }$, 
$U_{be}:=U_b\times U_b\times U_b\times U_h$,
\begin{subequations}
\label{eq:ObsCLsysExtIO_operators}
\eqn{
\mcA_e &= \pmat{\mcA &0&0&0 \\ 
-L_i \mcC & \mcA + L_i \mcC &0&0\\
0 &0& \mcA + L_i \mcC &0\\
-L_i \mcC & 0&0&\mcA + L_i \mcC 
}, \
B_{ei} = \pmat{ 0&0&0\\  0&0&0\\  0&0&0\\0& L_i&0}
\\
\mcB_e &= \pmat{ 
\mcB-L\CC & -\BB+L\CC&0&0 \\
L\mcC & \mcB  - L \mcC&Q(\BB_d-L_d\CC)&- Q\mcK\\0 &\BB-L\CC&Q(\BB_d-L_d\CC)&-\BB+L\CC\\  0&0&\BB_h-L_h\CC&0},\\
\label{eq:ObsCLsysExtIO_operators2}
Q_e &= \pmat{ 0 & 0&Q \\Q &0& 0\\0& -L&0\\ 0&0&0}, \quad 
\mcC_e = \pmat{\mcC & 0&0&0\\ 0&0&0& \mcC \\ 0&0&\BB_d-L_d\CC&0\\ 0&0&0&\KK}.
}
\end{subequations}
The functions $\phi_{es}:X_e\to U_e$ and $\phi_{eo}:Y_e\to U_e$ are defined by
$\phi_{es}(x_e)=(0,0,\phi_s(x))\tp$
and $\phi_{eo}(y_e)=(0,0,\phi_o(y))\tp$ for $x_e=(x,x_s,x_i,\hat x)\tp\in X_e$ and $y_e=(y,\hat y,\hat d,u_0)\tp\in Y_e$.
We make the following assumptions on the parameters of the system and the controller.
\begin{assumption}
\label{ass:ADRCass}
We assume that the parameters of the system~\eqref{eq:plant} and 
$\mcK\in \Lin(\Dom(\mcA),U)$, $L\in \Lin(Y,U_b)$, and $L_i \in \Lin(Y,X)$
have the following properties.
\begin{itemize}
\item[\textup{(a)}]
$(\mcB ,\mcA,  \pmatsmall{\mcC\\ \mcK}, [Q,L], [0,L_i])$ is a well-posed boundary node on $\BNsp{U_b}{U\times Y}{X}{Y\times U}$
 and~\eqref{eq:BQCstruct} hold;
\item[\textup{(b)}]
 $\phi_s:X\to U$ is locally Lipschitz continuous and has linear growth
  and $\phi_o: Y\to U$ is globally Lipschitz continuous;
\item[\textup{(c)}]
$(I-P_K(\cdot))\inv \in H_\infty(\C_{\gb}^+;\Lin(U))$ and $(I-P_L(\cdot))\inv \in H_\infty(\C_{\gb}^+;\Lin(Y))$ 
for some $\gb\in \R$,
where 
 $P_K$ and $P_L$ are the transfer functions of $(\mcB ,\mcA,   \mcK, Q,0)$ and $(\mcB ,\mcA,  \mcC,  L, L_i)$, respectively;
\item[\textup{(d)}]
 $(\BB_h-L_h\CC,\AA_d,\BB_d-L_d\CC,0,0)$ with $\AA_d=(\AA+L_i\CC)\vert_{\ker(\CC_d)}: \ker(\CC_d)\subset X\to X$ 
 is a boundary node on $\BNsp{ U_h}{\set{0}}{X}{U_d}$ with a well-posed output map.

\end{itemize}
\end{assumption}

Our first main result below establishes the existence of unique classical and generalised solutions of the closed-loop system
 and continuous dependence of its state and output on the initial state and input.
 It also introduces conditions for exponential closed-loop stability 
and for the successful estimation and rejection of the unknown disturbance $d\in \Lploc[2](0,\infty;U)$ and the nonlinear effects.
Here we denote $A_K:=\mcA\vert_{\ker(\mcB-Q\mcK)}$, $A_L:=(\mcA+L_i\mcC)\vert_{\ker(\mcB-L\mcC)}$, and $A_d:=(\mcA+L_i\mcC)\vert_{\ker(\mcB_h-L_h\mcC)\cap \ker(\CC_d)}$.
Recall that the notation $L_\gw^2(0,\infty;U)$ was defined in \cref{sec:notation}.

\begin{theorem}
\label{thm:ADRCmain}
Suppose that Assumption~\textup{\ref{ass:ADRCass}} holds.
Then for every $x_{e0}\in X_e$ and $u_e=(u_1,u_2,d)\tp\in \Lploc[2](0,\infty;U_e)$ the system~\eqref{eq:ObsCLsysExtIO} has a unique generalised solution $(x_e,u_e,y_e)$ satisfying $x_e(0)=x_{e0}$. If $x_{e0}\in \Dom(\AA_e)$ and $u_e\in \Hloc{1}(0,\infty;U_e)$ satisfy $\BB_e x_{e0}=Q_e(u_e(0)+\phi_{es}(x_{e0})+\phi_{eo}(\CC_e x_{e0}))$, then $(x_e,u_e,y_e)$ is a classical solution of~\eqref{eq:ObsCLsysExtIO}.
Moreover,
$A_K$, $A_L$, and $A_d$ generate 
 semigroups $T_K$, $T_L$ and $T_d$, respectively, on $X$. 

 If $T_K$, $T_L$, and $T_d$ are exponentially stable
and if $\gw_d>\gw_0(T_d)$ and $\gw, \gw_s>\max \set{\gw_0(T_K),\gw_0(T_L),\gw_0(T_d)}$, then 
there exist constants $M_d,M>0$ such that
for every initial condition $x_{e0}=(x_0,x_{s0},x_{i0},\hat x_0)\tp\in X_e$ and input $u_e=(u_1,u_2,d)\tp\in \Lploc[2](0,\infty;U_e)$ the generalised solution $(x_e,u_e,y_e)$ of~\eqref{eq:ObsCLsysExtIO} has the following properties:

\begin{itemize}
\setlength{\itemindent}{-2ex}
\item[\textup{(a)}]
We have
\ieq{
\displaystyle\norm{\hat d-\dtot}_{L^2_{\gw_d}(0,\infty)} \le M_d \norm{(x_0,x_{s0},x_{i0})\tp}_{X^3};
}
\item[\textup{(b)}] 
If $u_e\in L_{\gw}^2(0,\infty;U_e)$, then $(y,\hat y,u_0)\tp\in L_\gw^2(0,\infty;Y\times Y\times U)$ and
\eq{
\norm{x(t)} + \norm{\hat x(t)} &\le Me^{\gw_s t} \norm{x_{e0}}_{X_e} + Me^{\gw t }\norm{(u_1,u_2)\tp}_{L_{\gw}^2(0,\infty)}, \qquad t\ge0\\
 \norm{(y,\hat y,u_0)\tp}_{L^2_{\gw}(0,\infty)} &\le M \norm{x_{e0}}_{X_e} + \norm{(u_1,u_2)\tp}_{L_{\gw}^2(0,\infty)};
}
\item[\textup{(c)}] 
If $(u_1,u_2)\tp\in L^2(0,\infty;U\times Y)$ and 
 $d\in L^\infty(0,\infty;U)$
or
 $d\in L^2(0,\infty;U)$, 
 then 
\ieq{
\sup_{t\ge0}\; \norm{x_e(t)}_{X_e}<\infty.
}
\end{itemize}
\end{theorem}

\begin{proof}
We again denote $\dtot = d+\phi_s(x)+\phi_o(y)$.
We will
complete a state transformation which corresponds to
 rewriting the closed-loop system in the new state variable $z_e(t)=(z_1(t),z_2(t),z_3(t),z_4(t))\tp$, where
 $z_1(t):=x(t)-x_s(t)$, $z_2(t)=x(t)$, $z_3(t)= x(t)-\hat x(t)$, and $z_4(t)= x(t)-x_s(t)-x_i(t)$.
For the Hilbert space $Y_0=Y\times U_b\times (Y\times U)^2\times U$
 we will define 
and operator $S_o\in \Lin(Y_0,Y_e)$, boundedly invertible operators $S\in \Lin(X_e)$ and $S_b\in \Lin(U_{be})$, and a boundary node  
$( \BB_e', \AA_e',  \CC_e',Q_e', B_{ei}')$  on $\BNsp{U_{be}}{U_e}{X_e}{Y_0}$
in such a way that
 $z_e(t)=S x_e(t)$,
 $ \AA_e' = S\AA_e S\inv$ with domain $\Dom( \AA_e')=S( \Dom(\AA_e))$, 
 $ \BB_e' = S_b \BB_e S\inv$,
$ \CC_e=S_o\CC_e' S$,
 $ Q_e'=S_b Q_e$, and
$ B_{ei}'=SB_{ei}$. 
Once we show that 
$( \BB_e', \AA_e',  \CC_e',Q_e', B_{ei}')$ is a well-posed boundary node on $\BNsp{U_{be}}{U_e}{X_e}{Y_0}$,
\cref{lem:BnodeModifications} will imply
that 
$(\mcB_e,\mcA_e,\mcC_e,Q_e,B_{ei})$ is a well-posed boundary node on $\BNsp{U_{be}}{U_e}{X_e}{Y_e}$.

More precisely, we define
 $S_o\in \Lin(Y_0,Y_e)$, $S\in \Lin(X_e)$, and $S_b\in \Lin(U_{be})$ in the following way. Our change of coordinates implies that $Sx_e = (x-x_s,x,x-\hat x,x-x_s-x_i)\tp$ for $x_e=(x,x_s,x_i,\hat x)\tp\in X_e$, and $S\inv z_e=(z_2,z_2-z_1,z_1-z_4,z_2-z_3)\tp$ for $z_e=(z_1,z_2,z_3,z_4)\tp\in X_e$.
We define
 $S_b$
 and $S_o$ so that
\eq{
S_bu & = (u_1,u_2,u_3+u_1,u_4,u_5+u_1,u_6,-u_7-u_1)\tp\\
S_o y&=(y_3,y_3-y_5,-L_d y_{1d}+y_{2d} - y_7,y_4-y_6)\tp
}
for $u=(u_1,u_2,u_3,u_4,u_5,u_6,u_7)\tp\in U_{be}=(U_h\times U)^3\times U_h$ and
$y=( y_1,y_2,y_3,$ $y_4,y_5,y_6,y_7)\tp\in Y_0$ with $y_1=(y_{1h},y_{1d})\tp\in U_h\times U_d$ and $y_2=(y_{2h},y_{2d})\tp\in U_h\times U_d$.
Moreover, we define  
$\CC_e'\in \Lin(\Dom(\AA_e'),Y_0)$, $Q_e'\in \Lin(U_e,U_{be})$, $B_{ei}'\in \Lin(U_e,X_e)$ 
by
$Q_e':=S_bQ_e=Q_e$, $B_{ei}'u_e:=SB_{ei}u_e=(0,0,-L_iu_2,0)\tp$ for $u_e=(u_1,u_2,d)\tp\in U_e$,
and $\CC_e' z_e = (\CC z_1,\BB z_1,\CC z_2, \KK z_2, \CC z_3, \KK z_3,(\BB_d-L_d\CC)z_4)\tp \in Y_0$ for $z_e=(z_1,z_2,z_3,z_4)\tp\in \Dom(\AA_e')=S( \Dom(\AA_e))$.
 To compute $ \AA_e' = S\AA_e S\inv$ with domain $\Dom( \AA_e')=S( \Dom(\AA_e))$
we note that for $z_e =(z_1,z_2,z_3,z_4)\tp\in \Dom(\AA)^4$ we have
\eq{
z_e\in \Dom( \AA_e') \quad \Leftrightarrow \quad &S\inv z_e = (z_2,z_2-z_1,z_1-z_4,z_2-z_3)\tp \in \Dom(\AA_e)\\
\Leftrightarrow \quad & 
\CC_d (z_2-(z_2-z_1)-(z_1-z_4))=0\\
\Leftrightarrow \quad & 
\CC_d z_4=0.
}
A direct computation shows that for  $z_e =(z_1,z_2,z_3,z_4)\tp\in \Dom( \AA_e')$ we have
\eq{
 \AA_e' z_e
&= S \AA_e S\inv z_e 
= S \AA_e (z_2,z_2-z_1,z_1-z_4,z_2-z_3)\tp\\
&= S\pmat{\AA z_2\\\AA z_2-(\AA+L_i\CC)z_1\\(\AA+L_i\CC)(z_1-z_4)\\\AA z_2-(\AA+L_i\CC)z_3}
= \pmat{(\AA+L_i\CC)z_1\\\AA z_2\\(\AA+L_i\CC)z_3\\(\AA+L_i\CC)z_4}.
}
Thus $\AA_e' = \diag (\AA+L_i\CC,\AA,\AA+L_i\CC,\AA+L_i\CC)$ with domain $\Dom(\AA_e' ) = \Dom(\AA)^3\times \ker(\CC_d)$.
Similarly, a direct computation shows that 
$ \BB_e' :=  S_b\BB_e S\inv \in \Lin(\Dom(\AA_e'),U_{be})$ satisfies
\eq{
\MoveEqLeft[.5] \BB_e' z_e
=  S_b\BB_e (z_2,z_2-z_1,z_1-z_4,z_2-z_3)\tp\\
&= S_b\pmat{
(\mcB-L\CC)z_2 + (-\BB+L\CC)(z_2-z_1)\\
L\mcC z_2+ (\mcB- L \mcC)(z_2-z_1)  + Q(\BB_d-L_d\CC) (z_1-z_4)- Q\mcK (z_2-z_3)\\
(\BB-L\CC)(z_2-z_1)+Q(\BB_d-L_d\CC)(z_1-z_4)+ (-\BB+L\CC)(z_2-z_3)\\
(\BB_h-L_h\CC)(z_1-z_4)
}\\
&=\pmat{
\mcB-L\CC &0&0&0 \\
0& \BB-Q\KK& Q\mcK &- Q(\BB_d-L_d\CC) \\
0&0 &\BB-L\CC& -Q(\BB_d-L_d\CC) \\
0&0&0&\BB_h-L_h\CC
}
\pmat{z_1\\z_2\\z_3\\z_4}.
}
Finally,  for $x_e=(x,x_s,x_i,\hat x)\tp\in \Dom(\AA_e)$ we have
\eq{
 S_o \CC_e' S x_e
&=
S_o \CC_e' (x-x_s,x,x-\hat x,x-x_s-x_i)\tp
 = (\CC x, \CC\hat x, (\BB_d-L_d\CC)x_i, \KK\hat x)\tp \\
&= \CC_e x_e
}
and thus $\CC_e= S_o \CC_e' S $,
 as required.

We will now analyse the boundary node property and well-posedness of $( \BB_e', \AA_e',  \CC_e',Q_e', B_{ei}')$ and 
$( \BB_e, \AA_e,  \CC_e,Q_e, B_{ei})$.
By assumption, 
the boundary node $(\mcB ,\mcA,  \pmatsmall{\mcC\\ \mcK}, [Q,L], [0,L_i])$ is well-posed
and we denote its transfer function by $\pmatsmall{P&P_L\\ P_K& P_{KL}}$.
Since 
$(I-P_K(\cdot))\inv \in H_\infty(\C_{\gb}^+;\Lin(U))$ and $(I-P_L(\cdot))\inv \in H_\infty(\C_{\gb}^+;\Lin(Y))$,
applying \cref{prp:BCSfeedback} with the operators $K=\pmatsmall{0&I\\ 0&0}$ and $K=\pmatsmall{0&0\\ I&0}$ shows that $(\BB-Q\KK,\AA,\CC,Q,0)$ and $(\BB-L\CC,\AA+L_i\CC,\pmatsmall{\CC\\\KK},Q,0)$ are well-posed boundary nodes on \BNspStandard\ and $\BNsp{U_b}{U}{X}{Y\times U}$, respectively. In particular, $A_K=\AA\vert_{\ker(\BB-Q\KK)}$ and $A_L=(\AA+L_i\CC)\vert_{\ker(\BB-L\CC)}$ generate strongly continuous $T_K$ and $T_L$, respectively, and $A_d$ generates a semigroup $T_d$ by assumption.
We define 
 $\AA_1: \Dom(\AA_1)\subset X^2\to X^2$, $\BB_1\in \Lin(\Dom(\AA_1),U_b\times U_h)$, $\CC_1\in \Lin(\Dom(\AA_1),Y\times U^2)$, $Q_1\in \Lin(Y,U_b\times U_h)$, and $B_{1i}\in \Lin(Y,X^2)$ by
$\AA_1 = \diag(\AA+L_i\CC,\AA+L_i\CC)$ with domain 
$\Dom(\AA_1)= \Dom(\AA)\times \ker(\CC_d)$,
\eq{
\BB_1 = \pmat{\BB-L\CC& -Q(\BB_d-L_d\CC)\\0&\BB_h-L_h\CC},
\quad
\CC_1 = \pmat{\CC&0\\\KK&0\\ 0&\BB_d-L_d\CC}, 
\quad
Q_1 = \pmat{-L\\0},
}
and
 $B_{1i}=\pmatsmall{-L_i\\0}$.
Moreover, we define $K_1 = \pmat{0,\, -I,\, I}\in \Lin(Y\times U^2,U)$ and define
 $\AA_r: \Dom(\AA_r)\subset X^3\to X^3$, $\BB_r\in \Lin(\Dom(\AA_r),U_b^2\times U_h)$, $\CC_r\in \Lin(\Dom(\AA_r),(Y\times U)^2\times  U)$, 
$Q_r=\diag(Q,Q_1)\in \Lin(U\times Y,U_b^2\times U_h)$, 
$B_{ri}\in \Lin(U\times Y,X^3)$
such that $\AA_r = \diag (\AA,\AA_1)$ with domain $\Dom(\AA_r)=\Dom(\AA)\times \Dom(\AA_1)$, 
\eq{
\BB_r&=
\pmat{\BB-Q\KK&-QK_1\CC_1\\ 0& \BB_1} , \quad
\CC_r=
\pmat{
\CC&0\\ 
\KK&0\\
0&\CC_1
}
, \quad \mbox{and} \quad
B_{ri}=\pmat{0&0\\0&B_{i1}}.
}
Since $(\BB-L\CC,\AA+L_i\CC,\pmatsmall{\CC\\\KK},Q,0)$
and
$(\BB_h-L_h\CC,\AA_d,\BB_d-L_d\CC,0,0)$ with $\AA_d=(\AA+L_i\CC)\vert_{\ker(\CC_d)}: \ker(\CC_d)\subset X\to X$
are well-posed boundary nodes on  $\BNsp{U_b}{Y}{X}{Y\times U}$ and $\BNsp{ U_h}{\set{0}}{X}{U_d}$, respectively,
  \cref{prp:CascWP}
 implies that $(\BB_1,\AA_1,\CC_1,Q_1,B_{1i})$ is a well-posed boundary node on $\BNsp{U_b\times U_h}{U}{X^2}{Y\times U^2}$ and  that the semigroup $T_1$ generated by $\AA_1\vert_{\ker(\BB_1)}$ satisfies $\gw_0(T_1)= \max \set{\gw_0(T_L),\gw_0(T_d)}$.
Since 
$(\BB-Q\KK,\AA,\CC,Q,0)$ is a well-posed boundary node on \BNspStandard,
a second application of \cref{prp:CascWP}
shows that
$(\BB_r,\AA_r,\CC_r,Q_r,B_{ri})$ is a well-posed boundary node on $\BNsp{U_b^2 \times U_h}{U\times Y}{X^3}{(Y\times U)^2 \times U}$ and that
 $A_r=\AA_r\vert_{\ker(\BB_r)}$ generates a semigroup $T_r$ with $\gw_0(T_r)= \max \set{\gw_0(T_K),\gw_0(T_1)}=\max \set{\gw_0(T_K),\gw_0(T_L),\gw_0(T_d)}$.
Finally,
by construction we have $\AA_e'=\diag(\AA+L_i\CC,\AA_r)$ with domain $\Dom(\AA_e')=\Dom(\AA)\times \Dom(\AA_r)$,
\eq{
\BB_e' = \pmat{\BB+L\CC&0\\0&\BB_r}, 
\
\CC_e' = \pmat{\CC_0&0\\0&\CC_r}, 
\
Q_e' = \pmat{0&Q\\Q_r&0},
\
 B_{ei}'=\pmat{0&0\\B_{ri}&0},
}
where  $\CC_0\in \Lin(\Dom(\AA),Y\times U_b)$ is defined 
 by $\CC_0z_1=(\CC z_1,\BB z_1)\tp$ for $z_1\in \Dom(\AA)$.
Since 
$(\BB-L\CC,\AA+L_i\CC,\CC_0,Q,0)$ is a well-posed boundary node by \cref{lem:BnodeModifications}, this structure implies that 
$( \BB_e', \AA_e',  \CC_e',Q_e', B_{ei}')$ is a well-posed boundary node on $\BNsp{U_{be}}{U_e}{X_e}{Y_0}$, and by \cref{lem:BnodeModifications} $( \BB_e, \AA_e,  \CC_e,Q_e, B_{ei})$
is a well-posed boundary node on 
$\BNsp{U_{be}}{U_e}{X_e}{Y_e}$. Moreover, this structure and similarity of $T_e$ and $T_e'$ imply
\eq{
\gw_0(T_e)
=\gw_0(T_e')
= \max \set{\gw_0(T_L),\gw_0(T_r)}
= \max \set{\gw_0(T_K),\gw_0(T_L),\gw_0(T_d)}.
}

We will now show that the closed-loop system has unique classical and generalised solutions for all suitable initial states and inputs.
To this end, define $\phi_{es}'(z_e)=\phi_{es}(S\inv z_e)$ for $z_e\in X_e$ and $\phi_{eo}'(\tilde y_e)=\phi_{eo}(S_o \tilde y_e)$ for $\tilde y_e\in Y_0$.
Let $x_{e0}\in X_e$ and $u_e\in \Lploc[2](0,\infty;U_e)$.
The relationship between $(\BB_e,\AA_e,\CC_e,Q_e,B_{ei})$ and $(\BB_e',\AA_e',\CC_e',Q_e',B_{ei}')$ and 
\cref{lem:BCSnonlinSolEquivalence}
imply that~\eqref{eq:ObsCLsysExtIO} has a unique generalised solution $(x_e,u_e,y_e)$ satisfying $x_e(0)=x_{e0}$ if and only if 
\begin{subequations}
\label{eq:TransCLsys}
\eqn{
\dot{z}_e(t) &=  \AA_e' z_e(t) + B_{ei}' u_e(t)\\
\BB_e' z_e(t) &= Q_e'(u_e(t)+\phi_{es}'( z_e(t))+\phi_{eo}'(\tilde y_e(t)))\\
\tilde y_e(t) &= \CC_e' z_e(t),
}
\end{subequations}
has a unique generalised solution $(z_e,u_e, \tilde y_e)$ satisfying $z_e(0)=Sx_{e0}$ (note that $B_{ei}'\phi_{es}'= 0 $ and $B_{ei}'\phi_{eo}'=0$).
In this situation the solutions are related by $z_e= Sx_e$ and $y_e=S_o\tilde y_e$.
Denote $u_e=(u,d)\tp$ with $u=(u_1,u_2)\tp\in U\times Y$, 
$z_e=(z_1,z_r)\tp$ with $z_r=(z_2,z_3,z_4)\tp$, 
and $\tilde y_e = (y_0,y_r)\tp$ with $y_0\in Y\times U_b$ and $y_r=(y,\hat y,u_0,\hat u_0,e)\tp$.
Then the structures of the operators $\AA_e'$, $\BB_e'$, $\CC_e'$, $Q_e'$, and $B_{ei}'$ and of the functions $\phi_{es}'$ and $\phi_{eo}'$  show that~\eqref{eq:TransCLsys} has the form
\begin{subequations}
\label{eq:TransCLsysSplit}
\eqn{
\ddb{t} \pmat{z_1(t)\\z_r(t)} &=    \pmat{(\AA+L_i\CC) z_1(t) \\\mcA_r z_r(t)+ B_{ri}u(t)}\\
\pmat{(\BB-L\CC) z_1(t)  \\  \BB_r z_r(t)}
&= \pmat{Q(d(t)+\phi_s(K_sx_r(t))+\phi_o(K_oy_r(t)))\\ Q_ru(t)}
\\
\pmat{y_0(t)\\ y_r(t)} &= \pmat{\CC_0 z_1(t)\\ \CC_rz_r(t)},
} 
\end{subequations}
where  $\CC_0\in \Lin(\Dom(\AA),Y\times U_b)$
 is defined by $\CC_0z_1=(\CC z_1,\BB z_1)\tp$ for $z_1\in \Dom(\AA)$ and where
$K_s=[I,0,0]\in \Lin(X^3,X)$ and  $K_o=[I,0,0,0,0]\in \Lin( (Y\times U)^2\times U,Y)$.
We saw above that $(\BB_r,\AA_r,\CC_r,Q_r,B_{ri})$ and $(\BB-L\CC,\AA+L_i\CC,Q,\CC_0,0)$ are well-posed boundary nodes on $\BNsp{U_b^2 }{U\times Y}{X^3}{(Y\times U)^2\times  U}$ and
  $\BNsp{U_b }{U}{X}{Y\times U_b}$, respectively.
Finally,
$x_r\mapsto\phi_s(K_sz_r)$ is locally Lipschitz and has linear growth and $y_r\mapsto\phi_o(K_oy_r)$ is globally Lipschitz.
Therefore \cref{prp:BNcascadeNL} shows that~\eqref{eq:TransCLsys} has a unique generalised solution $(z_e,u_e,\tilde y_e)$ satisfying $z_e(0)=Sx_{e0}$. As argued above,  
 $(x_e,u_e,y_e)$ with $  x_e=S\inv z_e$ and $y_e=S_o\tilde y_e$ 
 is the unique generalised solution of~\eqref{eq:ObsCLsysExtIO} satisfying $x_e(0)=x_{e0}$.
If
 $x_{e0}=(x_0,x_{s0},x_{i0},\hat x_0)\tp\in \Dom(\AA_e)$ and $u_e=(u,d)\tp\in \Hloc{1}(0,\infty;U_e)$ satisfy $\BB_e x_{e0}=Q_e(u_e(0)+\phi_{es}(x_{e0})+\phi_{eo}(\CC_e x_{e0}))$,
then $z_{e0}=(z_{10},z_{20},z_{30},z_{40})\tp:=Sx_{e0}\in \Dom(\AA_e')$ satisfies
$(\BB-L\CC)z_{10}=Q(d(0)+\phi_s(K_sz_{r0})+\phi_o(K_o\CC_r z_{r0}))$ and $\BB_r z_{r0}\tp=Q_ru(0)\tp$ with $z_{r0}=(z_{20},z_{30},z_{40})\tp$, and thus \cref{prp:BNcascadeNL} implies that $(z_e,u_e,\tilde y_e)$ and $(x_e,u_e,y_e)$ are a classical solutions of~\eqref{eq:TransCLsys} and~\eqref{eq:ObsCLsysExtIO}, respectively.

To complete the proof, we will analyse the behaviour of the state trajectories and outputs of the closed-loop system, as well as the error $\dtot-\hat d$ of the disturbance estimation.
For this purpose, we assume that $T_K$, $T_L$, and $T_d$ are exponentially stable.
Our analysis above shows that
\ieq{
\gw_0(T_e)
= \gw_0(T_e')
= \max \set{\gw_0(T_K),\gw_0(T_L),\gw_0(T_d)}
= \gw_0(T_r)
}
and thus $T_e$, $T_e'$, and $T_r$ are exponentially stable.
The boundary conditions~\eqref{eq:plantBC} and~\eqref{eq:ADRCcontrxsBC} imply that for classical solutions $(x_e,u_e,y_e)$ of the closed-loop system~\eqref{eq:ObsCLsysExtIO} with $x_e= (x,x_s,x_i,\hat x)\tp$, $u_e=(u_1,u_2,d)$, and $y_e = (y,\hat y,\hat d,u_0)\tp$ we have
 $\dtot(t)=(\BB_d-L_d\CC)(x(t)-x_s(t))$, $t\ge 0$. Thus 
the error $\dtot-\hat d $ of the disturbance estimation satisfies
\eq{
\dtot(t)-\hat d(t)
&= (\BB_d-L_d\CC)(x(t)-x_s(t)) - (\BB_d-L_d\CC)x_i(t) \\
&= (\BB_d-L_d\CC)z_4(t)
}
for $t\ge0$, where $z_4=T_d(\cdot)(x(0)-x_s(0)-x_i(0))$.
This and our assumptions on $\phi_o$ and $\phi_s$ further imply that 
 $\dtot-\hat d= \Psi^d_\infty (x(0)-x_s(0)-x_i(0))$ also for generalised solutions $(x_e,u_e,y_e)$ of the closed-loop system, 
where $\Psi_\infty^d$ is the extended output map of the well-posed boundary node $(\BB_h-L_h\CC,\AA_d,\BB_d-L_d\CC,0,0)$.
Since $T_d$ is exponentially stable by assumption, for any $\gw_d>\gw_0(T_d)$ we have from \cref{prp:BNstability} that $\norm{\dtot-\hat d}_{L_{\gw_d}^2(0,\infty)}\le M_d \norm{(x(0),x_s(0),x_i(0))\tp}$ for a constant $M_d>0$ independent of $x_e(0)$ and $u_e$.
If $(x_e,u_e,y_e)$ is a generalised solution of the closed-loop system and with $x_e=(x,x_s,x_i,\hat x)\tp$, $u_e=(u_1,u_2,d)\tp$, and $y_e=(y,\hat y,u_0)\tp$, 
then our analysis and the structure of~\eqref{eq:CLsysReduced}
show that $(z_r,u,\tilde y_r)$ with $z_r=(x,x-\hat x,x-x_s-x_i)\tp$, $u=(u_1,u_2)\tp$, and $\tilde y_r=(y,\hat y, u_0)\tp$ is a generalised solutions of 
\begin{subequations}
\label{eq:CLsysReduced}
\eqn{
\dot z_r(t) &= \AA_r z_r(t) + B_{ri}u(t)\\
\BB_r z_r(t) &= Q_r u(t)\\
\tilde y_r(t) &= S_{ro}\CC_r z_r(t),
}
\end{subequations}
where  $S_{ro}\in \Lin((Y\times U)^2\times U,Y^2\times U)$ is defined by $S_{ro}y=(y_1,y_1-y_3,y_2-y_4)\tp$ for $y=(y_1,y_2,y_3,y_4,y_5)\tp\in (Y\times U)\times U$.
Since $(\BB_r,\AA_r,S_{ro}\CC_r,Q_r,B_{ri})$ is a well-posed boundary node by \cref{lem:BnodeModifications} and $T_r$ is exponentially stable, 
the existence of $M>0$ in the estimates for $\norm{x(t)}$, $\norm{\hat x(t)}$, and $\norm{(y,\hat y,u_0)\tp}_{L^2_\gw(0,\infty)}$ in (b) follows from \cref{prp:BNstability}. 
Finally, if 
$(u_1,u_2)\tp=u\in L^2(0,\infty; U\times Y)$,
then \cref{prp:BNstability} shows that $\sup_{t\ge 0}\norm{z_r(t)}<\infty$ and $\tilde y_r=(y,\hat y,u_0)\tp\in L^2(0,\infty;Y^2\times U)$. Since $K_sz_r=x$ and $K_oy_r=y$ in~\eqref{eq:TransCLsysSplit},
our assumptions on $\phi_s$ and $\phi_o$ imply that $\phi_s(K_sz_r)+\phi_o(0)\in L^\infty(0,\infty;U)$ and $\phi_o(K_oy_r)-\phi_o(0)\in L^2(0,\infty;U)$.
The structure of~\eqref{eq:TransCLsysSplit},
the stability of $T_L$, and the well-posedness of $(\BB-L\CC,\AA+L_i\CC , Q,\CC_0,0)$ together
 with Propositions~\ref{prp:WPBCSsolutions} and~\ref{prp:BNstability} imply that if $d\in L^2(0,\infty;U)$ or $d\in L^\infty(0,\infty;U)$, then $\sup_{t\ge 0}\norm{z_1(t)}<\infty$. Thus $\sup_{t\ge 0}\norm{x_e(t)}\le \norm{S\inv}\sup_{t\ge 0}\norm{z_e(t)}<\infty$ as claimed.
\end{proof}

\begin{remark}
\label{rem:CLstatebddcond}
The proof of \cref{thm:ADRCmain} shows that if $\phi_o=0$, we also have
$\sup_{t\ge0}\; \norm{x_e(t)}_{X_e}<\infty$ in (c)
under the alternative assumption 
$(u_1,u_2)\tp\in L^\infty(0,\infty;U\times Y)$.
Indeed, \cref{prp:BNstability} then implies that
$\sup_{t\ge 0}\norm{z_r(t)}<\infty$, 
 $\phi_s(K_sz_r(t))\in L^\infty(0,\infty;U)$, and $\sup_{t\ge 0}\norm{z_1(t)}<\infty$
in the last part of the proof.
\end{remark}

The following result presents additional properties of the closed-loop system.

\begin{proposition}
Suppose that Assumption~\textup{\ref{ass:ADRCass}} holds.
Then $(\mcB_e,\mcA_e,\mcC_e,Q_e,$ $B_{ei})$ is a well-posed boundary node on $\BNsp{U_{be}}{U_e}{X_e}{Y_e}$ and
 the semigroup $T_e$ generated by $A_e:= \mcA_e\vert_{\ker(\mcB_e)}$ satisfies
\ieq{
\gw_0(T_e)=\max \set{\gw_0(T_K),\gw_0(T_L), \gw_0(T_d)}.
}
For every $x_{e0}\in X_e$, $u_e=(u_1,u_2,d)\tp\in \Lploc[2](0,\infty;U_e)$, 
the tuple $(x_e,u_e,y_e)$ is the generalised solution of~\eqref{eq:ObsCLsysExtIO} satisfying $x_e(0)=x_{e0}$ 
in Theorem~\textup{\ref{thm:ADRCmain}}
if and only if 
$x_e\in C(\zinf;X_e)$ and $y_e\in \Lploc[2](0,\infty;Y_e)$
are the unique functions 
satisfying
\begin{subequations}
\label{eq:CLmildsoleq}
\eqn{
x_e(t) &= T_e(t)x_{e0} + \Phi_t^e \Pt (u_e + \phi_{es}(x_e)+ \phi_{eo}(y_e)), \qquad t\ge 0\\
y_e &= \Psi_\infty^e x_{e0} + \F_\infty^e  (u_e + \phi_{es}(x_e)+ \phi_{eo}(y_e)),
}
\end{subequations}
where $(T_e,\Phi^e,\Psi^e,\F^e)$ is the well-posed linear system 
associated to $(\mcB_e,\mcA_e,\mcC_e,$ $Q_e,B_{ei})$.
For every $\tau >0$ and $R>0$ there exists $M_{\tau,R}>0$ such that every pair $(x_e,u_e,y_e)$ and $(x_e',u_e',y_e')$ of  generalised solutions of~\eqref{eq:ObsCLsysExtIO} satisfies
\eq{
\MoveEqLeft\norm{\Pt[\tau](x_e-x_e')}_\infty + 
\norm{\Pt[\tau](y_e-y_e')}_{L^2} \\
&\le 
M_{\tau,R}\left(
\norm{x_e(0)-x_e'(0)}_X + 
\norm{\Pt[\tau](u_e-u_e')}_{L^2} 
  \right) 
}
whenever 
$\norm{x_e(0)}\le R$, 
$\norm{x_e'(0)}\le R$, 
$\norm{\Pt[\tau]u_e}\le R$, 
and $\norm{\Pt[\tau]u_e'}\le R$.
If $\phi_s$ is globally Lipschitz, $M_{\tau,R}$ can be chosen independently of $R>0$.
\end{proposition}

\begin{proof}
The boundary node property and well-posedness of
$(\mcB_e,\mcA_e,\mcC_e,Q_e,$ $B_{ei})$ as well as the identity
$\gw_0(T_e)=\max \set{\gw_0(T_K),\gw_0(T_L), \gw_0(T_d)}$ were established in the proof of \cref{thm:ADRCmain}. 
To prove the second part we note that \cref{prp:BCSNLsolutions} implies that
the tuple $(x_e,u_e,y_e)$ is a generalised solution of~\eqref{eq:ObsCLsysExtIO} satisfying $x_e(0)=x_{e0}$ 
if and only if 
$x_e\in C(\zinf;X_e)$ and $y_e\in \Lploc[2](0,\infty;Y_e)$
are such that~\eqref{eq:CLmildsoleq} hold. Since the generalised solution of~\eqref{eq:ObsCLsysExtIO} is unique by \cref{thm:ADRCmain}, equations~\eqref{eq:CLmildsoleq} can be satisfied for exactly one pair
$x_e\in C(\zinf;X_e)$ and $y_e\in \Lploc[2](0,\infty;Y_e)$ of functions.

Let $R,\tau>0$  and let $(x_e,u_e, y_e)$ and $(x_e',u_e', y_e')$ be two generalised solutions
of~\eqref{eq:ObsCLsysExtIO} satisfying 
$\norm{x_e(0)}\le R$, 
$\norm{x_e'(0)}\le R$, 
$\norm{\Pt[\tau]u_e}\le R$, 
and $\norm{\Pt[\tau]u_e'}\le R$.
The arguments in the proof of \cref{thm:ADRCmain} imply that 
 the system~\eqref{eq:TransCLsys} has generalised solutions 
$(z_e,u_e,\tilde y_e)$ and $(z_e',u_e',\tilde y_e')$
such that $z_e=Sx_e$, $z_e'=S_ex_e'$, $y_e=S_o\tilde y_e$, and $y_e'=S_o\tilde y_e'$.
Note that $\norm{S}\le 3$ and $\norm{S\inv}\le 2$.
Since
$\norm{z_e(0)}\le \norm{S} R\le 3R$ and
$\norm{z_e'(0)}\le \norm{S}R\le 3R$, the 
arguments in the proof of \cref{thm:ADRCmain}  and \cref{prp:BNcascadeNL} imply that there exists $M_{\tau,R}'>0$ depending on $\tau$ and $R$ such that
\eq{
\MoveEqLeft\norm{\Pt[\tau](x_e-x_e')}_\infty + 
\norm{\Pt[\tau]( y_e- y_e')}_{L^2} \\
& \le \norm{S\inv}\norm{\Pt[\tau](z_e-z_e')}_\infty + 
\norm{S_o}\norm{\Pt[\tau](\tilde y_e-\tilde y_e')}_{L^2} \\
&\le M_{\tau,R}'\max \set{3,\norm{S_o}}\left( \norm{x_e(0)-x_e'(0)}_X + \norm{\Pt[\tau](u_e-u_e')}_{L^2} \right) .
}
This estimate implies the claim with $M_{\tau,R}=M_{\tau,R}'\max \set{3,\norm{S_o}}$.
Moreover, if $\phi_s$ is globally Lipschitz,  $M_{\tau,R}'$ (and thus also $M_{\tau,R}$) can be chosen independently of $R>0$. 
\end{proof}

\begin{remark}
\label{rem:Adsyssuffcond}
Here we introduce a sufficient condition for 
Assumption~\ref{ass:ADRCass}(d).
Assume that $(\mcB ,\mcA,  \mcC, I,0)$ is a well-posed boundary node on $\BNsp{U_b}{U_b}{X}{Y}$ with $U_b=U_h\times U$ and $Y=U\times U_h$ satisfying
$\re \iprod{\AA x}{x}_X \le \re \iprod{\BB x}{\pmatsmall{0&I\\I&0}\CC x}_{U_b} $ for 
$ x\in \Dom(\AA).  $
The transfer function of $(\mcB ,\mcA,  \mcC, I,0)$ has the form
 $ \pmatsmall{P_{dh}&P_{dd}\\ P_{mh}& P_{md}}$.
If 
 $P_{dd}(\cdot)\inv\in H_{\infty}(\C_\gb^+;\Lin(U,Y_d))$ for some $\gb>0$,
if $L=\pmatsmall{0&L_0\\ 0&0}$ with $L_0\le 0$, and if $L_i=0$,
then~\citel{NicPau25}{Prop.~2.8 \& Rem.~2.9}, \cref{lem:IPBCSWP}, \citel{Pau19}{Lem.~A.2}, and \cref{prp:BCSfeedback} can be used   to show that
$(\BB_h-L_h\CC,\AA_d,\BB_d-L_d\CC,0,0)$ 
 is a boundary node on $\BNsp{ U_h}{\set{0}}{X}{U}$ with a well-posed output map.
\end{remark}

There are important systems
 for which the inverse system $(\BB_h-L_h\CC,\AA_d,$ $\BB_d-L_d\CC,0,0)$ does not have a well-posed output map
 but which can still be controlled using active disturbance rejection.
The following result shows that  
 Assumption~\ref{ass:ADRCass}(d)
can be replaced with a slightly more complicated condition which can nevertheless be verified especially for many parabolic PDEs.
If Assumption~\ref{ass:ADRCass}(d) fails, $\hat d$ does not
 correspond to a  well-posed output map of the closed-loop system.
We again denote $A_d=(\mcA+L_i\mcC)\vert_{\ker(\mcB_h-L_h\mcC)\cap \ker(\CC_d)}$ with domain  $\Dom(A_d)=\ker(\mcB_h-L_h\mcC)\cap \ker(\CC_d)$.

\begin{theorem}
\label{thm:ADRCParab}
Suppose that parts \textup{(a)--(c)} of Assumption~\textup{\ref{ass:ADRCass}} are satisfied, 
that $(\BB_1,\AA_1,\CC_1,0,0)$ with 
$\Dom(\AA_1)=\Dom(\AA)\times \ker(\CC_d)$,
\eq{
\AA_1= \pmat{\AA&0\\0&\AA+L_i\CC}, \quad
\BB_1= \pmat{\BB&-Q(\BB_d-L_d\CC)\\0&\BB_h-L_h\CC},
 \quad
\CC_1 = \pmat{\CC&0\\\KK&0},
}
is a boundary node on $\BNsp{U_b\times U_h}{\set{0}}{X^2}{Y\times U}$ with a well-posed output map, and that $\gl \mapsto \norm{H(\gl)}\norm{(\BB_d-L_d\CC)(\gl-A_d)\inv}$ is uniformly bounded on $\C_\gb^+$ for some $\gb\in \R$.
Then the claims regarding existence and uniqueness of solutions in Theorem~\textup{\ref{thm:ADRCmain}} hold for the closed-loop system with a modified output $y_{e0}=(y,\hat y, u_0)\tp$, and
 $A_K$, $A_L$, and $A_d$ generate
 semigroups $T_K$, $T_L$, and $T_d$, respectively, on $X$. 
If $T_K$, $T_L$, and $T_d$ are exponentially stable, then for every
 $\gw, \gw_s>\max \set{\gw_0(T_K),\gw_0(T_L),\gw_0(T_d)}$ 
there exists a constant $M>0$ such that
\textup{(b)} and \textup{(c)} in Theorem~\textup{\ref{thm:ADRCmain}} hold.
\end{theorem}

\begin{proof}
To compensate for the modified output $y_{e0}=(y,\hat y,u_0)\tp$ in the closed-loop system we redefine $Y_e=Y^2\times U$ and $Y_0=Y\times U_b\times (Y\times U)^2 $,
and redefine 
$\CC_e\in \Lin(\Dom(\AA_e),Y_e)$,
$\CC_e'\in \Lin(\Dom(\AA_e'),Y_0)$ and $S_o\in \Lin(Y_0,Y_e)$ in the proof of \cref{thm:ADRCmain} by
$\CC_e x_e=(\CC x,\CC\hat x,\KK \hat x)\tp  \in Y_e$,
$\CC_e' z_e = (\CC z_1,\BB z_1,\CC z_2, \KK z_2, \CC z_3,$ $ \KK z_3)\tp \in Y_0$, and
$S_o y=(y_3,y_3-y_5,y_4-y_6)\tp\in Y_e$
 for $x_e=(x,x_s,x_i,\hat x)\tp\in \Dom(\AA_e)$,
$z_e=(z_1,z_2,z_3,z_4)\tp\in \Dom(\AA_e')=S( \Dom(\AA_e))$, 
and 
$y=( y_1,y_2,y_3,y_4,y_5,y_6)\tp\in Y_0$.
These operators satisfy $\CC_e=S_o\CC_e'S$.

Part (d) of Assumption~\ref{ass:ADRCass} was used in the proof of \cref{thm:ADRCmain} only in showing that 
$(\BB_r,\AA_r,\CC_r,Q_r,B_{ri})$ is a well-posed boundary node, in computing the growth bound $\gw_0(T_r)$ of the semigroup $T_r$ generated by $A_r:=\AA_r\vert_{\ker(\BB_r)}$, and in analysing $\dtot-\hat d$ (which is not required here).
We will redo the first two of these steps under the current assumptions
for the modified system corresponding to our new closed-loop output.
More precisely, we define
$\AA_r: \Dom(\AA_r)\subset X^3\to X^3$,
 $\BB_r\in \Lin(\Dom(\AA_r),U_b^2\times U_h)$, $\CC_r\in \Lin(\Dom(\AA_r),(Y\times U)^2)$, $B_{ri}\in \Lin(U\times Y,X^3)$, 
such that $\AA_r = \diag (\AA,\AA+L_i\CC,\AA+L_i\CC)$ with domain  $\Dom(\AA_r)=\Dom(\AA)\times \Dom(\AA)\times \ker(\CC_d)$,
\eq{
\BB_r&=
\pmat{\BB-Q\KK&Q\KK&-Q(\BB_d-L_d\CC)\\ 0&\BB-L\CC& -Q(\BB_d-L_d\CC)\\ 0&0&\BB_h-L_h\CC} , 
\quad
Q_r=\pmat{Q&0\\0&-L\\0&0},
}
$B_{ri}u=(0,-L_iu_2,0)\tp$ for $u=(u_1,u_2)\tp\in U\times Y$,
and $\CC_rz_r=(\CC z_2,\KK z_2,\CC z_3,\KK z_3)$ for $z_r=(z_2,z_3,z_4)\tp\in \Dom(\AA_r)$.
To prove well-posedness of $(\BB_r,\AA_r,\CC_r,Q_r,B_{ri})$, we further define
$\AA_{r0}: \Dom(\AA_{r0})\subset X^3\to X^3$ and
 $\BB_{r0}\in \Lin(\Dom(\AA_{r0}),U_b^2\times U_h)$ 
such that $\AA_{r0} = \diag (\AA,\AA,\AA+L_i\CC)$ with domain $\Dom(\AA_{r0})=\Dom(\AA_r)$ and
\eq{
\BB_{r0}&=
\pmat{\BB&0&-Q(\BB_d-L_d\CC)\\ 0&\BB& -Q(\BB_d-L_d\CC)\\ 0&0&\BB_h-L_h\CC} .
}
It is straightforward to use  Remark~\ref{rem:WPtestclass}, the well-posedness of  $(\BB_1,\AA_1,\CC_1,0,0)$ and $(\BB,\AA,\pmatsmall{\CC\\\KK},[Q,L],[0,L_i])$ and linearity
to show that 
  $(\BB_{r0},\AA_{r0},\CC_r,Q_r,B_{ri})$ is a well-posed boundary node on $\BNsp{U_b^2\times U_h}{U\times Y}{X^3}{(Y\times U)^2}$.
Denote the transfer function of 
$(\BB,\AA,\pmatsmall{\CC\\\KK},[Q,L],[0,L_i])$ by $\pmatsmall{P&P_L\\ P_K& P_{KL}}$. 
If we define $K\in \Lin( (U\times Y)^2,U\times Y)$ by $Ky=(y_2-y_4,-y_3)\tp$ for $y=(y_1,y_2,y_3,y_4)\tp\in (U\times Y)^2$, then $\AA_r=\AA_{r0}+B_iK\CC_r$ and $\BB_r=\BB_{r0}-QK\CC_r$.
Moreover, the transfer function $P_{r0}$ of $(\BB_{r0},\AA_{r0},\CC_r,Q_r,B_{ri})$ satisfies 
\eq{
I-K P_{r0}(\gl) 
= \pmat{I-P_{K}(\gl)&P_{KL}(\gl)\\0& I-P_L(\gl)}
}
for $\gl\in \rho(A)\cap \rho(A_d)$. Thus our assumptions on $P_K$ and $P_L$,~\citel{FkiPau26arxiv}{Rem.~2.6}, and \cref{prp:BCSfeedback} imply that 
$(\BB_r,\AA_r,\CC_r,Q_r,B_{ri})$ is a well-posed boundary node on $\BNsp{U_b^2\times U_h}{U\times Y}{X^3}{(Y\times U)^2}$.
We will now show that 
$\gw_0(T_r)= \max \set{\gw_0(T_K),\gw_0(T_L),\gw_0(T_d)}$.
Similarly as in the proof of~\citel{FkiPau26arxiv}{Prop.~2.11}, the resolvent operator of $A_r= \AA_r\vert_{\ker(\BB_r)}$ is given by
\eq{
(\gl-A_r)\inv = \pmat{R_K(\gl) \hspace{-.8ex}& -H_K(\gl)\KK R_L(\gl)& H_K(\gl)(I-P_{KLQ}(\gl)) \BB_{dL} R_d(\gl) \\ 0& R_L(\gl) &  H_{LQ}(\gl) \BB_{dL} R_d(\gl)  \\ 0&0& R_d(\gl)}
}
for $\gl\in \rho(A_K)\cap \rho(A_L)\cap \rho(A_d)$,
where we have denoted $\BB_{dL}=\BB_d-L_d\CC$,
$R_K(\gl)=(\gl-A_K)\inv $,
$R_L(\gl)=(\gl-A_L)\inv $,
$R_d(\gl)=(\gl-A_d)\inv $,
 where
$H_K$ is the transfer function of the well-posed boundary node $(\BB-Q\KK,\AA,\CC,Q,0)$, and 
where 
$H_{LQ}$ and $P_{KLQ}$ are the transfer functions of the well-posed boundary node $(\BB-L\CC,\AA+L_i\CC,\KK,Q,0)$.
This structure and the Gearhart--Pr\"uss--Greiner theorem immediately imply that 
$\gw_0(T_r)\ge \max \set{\gw_0(T_K),\gw_0(T_L),\gw_0(T_d)}$.
To show the converse estimate, we will show that for any
$\gw> \max \set{\gw_0(T_K),\gw_0(T_L),\gw_0(T_d)}$ the norms $\norm{(\gl-A_r)\inv}$ are uniformly bounded with respect to $\gl\in\C_\gw^+$.
In the block operator the norms $\norm{R_K(\gl)}$, $\norm{R_L(\gl)}$,  $\norm{R_d(\gl)}$,   $\norm{H_K(\gl)}$, $\norm{\KK R_L(\gl)}$, and $\norm{P_{KLQ}(\gl)}$ are uniformly bounded with respect to $\gl\in\C_\gw^+$ by~\citel{FkiPau26arxiv}{Rem.~2.6}.
Thus it suffices to show that 
$ \norm{H_K(\gl)}\norm{\BB_{dL}R_d(\gl)}$ and
$ \norm{H_{LQ}(\gl)}\norm{\BB_{dL}R_d(\gl)}$ are uniformly bounded with respect to $\gl\in \C_\gw^+$.
Let $\gd>0$ be such that $\C_{\gw+\gd}^+\subset \rho(A)$, 
$\gw+\gd>\gb$, 
 $(I-P_K(\cdot))\inv \in H_{\infty}(\C_{\gw+\gd}^+;\Lin(U))$, and 
 $(I-P_L(\cdot))\inv \in H_{\infty}(\C_{\gw+\gd}^+;\Lin(Y))$.
The resolvent identity and~\citel{NicPau25}{Prop.~2.4(iii)} imply%
\eq{
H_K(\gl)\BB_{dL}R_d(\gl)
= [I+\gd R_K(\gl)]H_K(\gl+\gd)\BB_{dL}R_d(\gl+\gd)[I+\gd R_d(\gl)]
}
for all $\gl\in \C_\gw^+$. 
The arguments in the proof of~\citel{NicPau25}{Prop.~2.6} show that $H_K(\gl+\gd)=H(\gl+\gd)(I-P_K(\gl+\gd))\inv$,
 and our assumptions imply that 
$ \norm{H_K(\gl)}\norm{\BB_{dL}R_d(\gl)}$ is indeed uniformly bounded with respect to $\gl\in \C_\gw^+$.
Analogously, 
uniform boundedness of 
$ \norm{H_{LQ}(\gl)}\norm{\BB_{dL}R_d(\gl)}$
reduces to the uniform boundedness of 
$ \norm{H_{LQ}(\gl+\gd)}\norm{\BB_{dL}R_d(\gl+\gd)}$, and similar arguments as in the proof of~\citel{NicPau25}{Prop.~2.6} show that $H_{LQ}(\mu)=H(\mu)+H_L(\mu)(I-P_L(\mu))\inv P(\mu)$ for $\mu\in \C_{\gw+\gd}^+$, where $H_L$ and $P_L$ are the transfer functions of $(\BB,\AA,\CC,L,L_i)$.
This together with our assumptions finally implies that $\norm{(\gl-A_r)\inv}$ is uniformly bounded with respect to $\gl\in\C_\gw^+$ and completes the proof that
$\gw_0(T_r)= \max \set{\gw_0(T_K),\gw_0(T_L),\gw_0(T_d)}$.

Since 
$(\BB-L\CC,\AA+L_i\CC,\KK,Q,0)$ is a well-posed boundary node, the arguments in the proof of \cref{thm:ADRCmain} now show that
$( \BB_e', \AA_e',  \CC_e',Q_e', B_{ei}')$ is a well-posed boundary node on $\BNsp{U_{be}}{U_e}{X_e}{Y_0}$,
 and by \cref{lem:BnodeModifications} $( \BB_e, \AA_e,  \CC_e,Q_e, B_{ei})$
is a well-posed boundary node on 
$\BNsp{U_{be}}{U_e}{X_e}{Y_e}$. Moreover, this structure and similarity of $T_e$ and $T_e'$ imply 
\ieq{
\gw_0(T_e)
=\gw_0(T_e')
= \max \set{\gw_0(T_L),\gw_0(T_r)}
= \max \set{\gw_0(T_K),\gw_0(T_L),\gw_0(T_d)}.
}
The rest of the arguments in the proof (without the analysis of $\dtot -\hat d$) establish existence and uniqueness of the solutions of the closed-loop system with output $(y,\hat y,u_0)$ and the estimates in parts (b) and (c) of \cref{thm:ADRCmain} for generalised solutions.
\end{proof}

\section{Disturbance Rejection and Stabilisation for PDE Systems}
\label{sec:PDEs}

\subsection{A Wave Equation}
We consider disturbance rejection and stabilisation for 
an undamped wave equation on the interval $(0,1)$, 
\begin{subequations}
\label{eq:Wave1D}
\eqn{
\ddot w(\xi,t) &= w''(\xi,t)\\
w'(0,t) &= 0 , \quad  w'(1,t) = u(t)+d(t)+\phi(\dot w(0,t))\\
y(t) &= \pmat{\dot w(1,t)\\ \dot w(0,t)}
}
\end{subequations}
for $\xi\in (0,1)$ and $t>0$
with initial conditions $w(\cdot,0)=  w_0\in H^1(0,1)$ and $\dot w(\cdot,0) = w_1\in L^2(0,1)$. Here $d\in \Lploc[2](0,\infty)$ is an unknown input disturbance and the term $\phi(\dot w(0,t))$ with an unknown globally Lipschitz function $\phi:\C\to \C$ represents an unmodeled nonlinearity.
Our aim is to reject the unknown boundary disturbance, compensate for the nonlinearity, and achieve exponential convergence of the energy of the solutions to zero. 
Active disturbance rejection for one-dimensional wave equations has been previously considered in~\cite{GuoJin13,GuoJin15,FenGuo17b,ZhoWei18,MeiZho21}.
We formulate~\eqref{eq:Wave1D} as a boundary control system of the form~\eqref{eq:plant} on $X=\Lp[2](0,1)\times L^2(0,1)$ with the state variable 
$x(t)=(w'(\cdot,t),\dot w(\cdot,t))\tp$.
When we choose the norm 
$\norm{\cdot}_X$ on $X$ to be the standard Hilbert space norm, the total energy of the 
 solution of the wave equation is  $E(t)=\frac12\norm{x(t)}_X^2$, $t\ge 0$.
We let $U=\C$, $U_b=\C^2$, and $Y=\C^2$ and
define  $\AA: \Dom(\AA)\subset X\to X$ with domain $\Dom(\AA)=H^1(0,1)\times H^1(0,1)$, $\BB\in \Lin(\Dom(\AA),U_b)$ and $\CC\in \Lin(\Dom(\AA),Y)$ by
\eq{
\AA x=\pmat{g'\\ f'}, \qquad \BB x= \pmat{-f(0)\\ f(1)}, \qquad \CC x= \pmat{g(1)\\ g(0)}
}
for $x=(f,g)\tp\in \Dom(\AA)$.
Moreover, we define $Q=\pmatsmall{0\\ 1}\in \Lin(\C,\C^2)$ and $B_i=0$. 
Then the operators $\BB$ and $\CC$ have the structure in~\eqref{eq:BQCstruct} with $U_h=\C$, $Y_d=\C$, and $Y_m=\C$ and with $\BB_h x= f(0) $, $\BB_d x=f(1)$, $\CC_dx =g(1)$, and $\CC_m x=g(0)$
for $x=(f,g)\tp\in \Dom(\AA)$.
Similar analysis as in~\citel{NicPau25}{Sec.~4.1} shows that $(\BB,\AA,\CC,Q,B_i)$ is a well-posed boundary node on \BNspStandard.

The 
solution $w$ of the original wave equation can be recovered from the classical solution  of the associated abstract system by integrating the state components $w'$ and $\dot w$ with respect to $\xi$ and $t$, respectively.
In the following we consider generalised solutions which are defined as limits of classical solutions of the abstract system as in \cref{def:Solutionsnonlin}.
Our Active Disturbance Rejection Controller has the form
\eqn{
\label{eq:Wave1Dcontr}
\begin{cases}
\ddot w_s(\xi,t) = w_s''(\xi,t), \quad 
\ddot w_i(\xi,t) = w_i''(\xi,t), \quad
 \ddot{\hat{w}}(\xi,t) = \hat w''(\xi,t)\\
w_s'(0,t) =  \dot w_s(0,t)-\ell \dot w(0,t)\\
w_s'(1,t) = -\kappa \dot{\hat{w}}(1,t) -w_i'(1,t)  + u_1(t)\\
w_i'(0,t)  =\ell \dot w_i(0,t), \qquad 
\dot w_i(1,t) = \dot w(1,t) - \dot w_s(1,t)\\
\hat w'(0,t)  =  \ell  \dot{\hat{w}}(0,t)+\ell u_3(t) - \dot w(0,t)\\
\hat w'(1,t) = -\kappa \dot{\hat{w}}(1,t) + u_1(t)\\
\hat{d}(t) = w_i'(1,t) \\
u(t)  = -\kappa \dot{\hat{w}}(1,t) -w_i'(1,t)  + u_1(t).
\end{cases}
}
The closed-loop state is $x_e=(x,x_s,x_i,\hat x)\tp\in X_e=L^2(0,1)^8$ with
$x_s(t)=(w_s'(\cdot,t),\dot w_s(\cdot,t))\tp$,
$x_i(t)=(w_i'(\cdot,t),\dot w_i(\cdot,t))\tp$, and
$\hat x(t)=(\hat w'(\cdot,t),\dot{\hat{w}}(\cdot,t))\tp$
 and its output is $y_e=(y,\hat y,\hat d)\tp =(\dot w(1,\cdot),\dot w(0,\cdot),\dot{\hat{w}}(1,\cdot),\dot{\hat{w}}(0,\cdot),w_i'(1,\cdot))\tp$.
 Note that we do not include  $u_0=-\kappa\dot{\hat{w}}(1,\cdot)$ as an output since 
it can be obtained from $\hat y$.
Our main result below shows that the controller~\eqref{eq:Wave1Dcontr} rejects the disturbance and stabilises the wave system~\eqref{eq:Wave1D}.

\begin{proposition}
\label{prp:Wave1Dmain}
Assume that $\phi:\C\to \C$ is globally Lipschitz continuous and let $\kappa>0$ and $\ell>0$.
For any initial conditions 
\eq{
x_{e1}:=(w(\cdot,0),w_s(\cdot,0),w_i(\cdot,0),\hat w(\cdot,0))\tp&\in H^1(0,1)^4\\
x_{e2}:=(\dot w(\cdot,0),\dot w_s(\cdot,0),\dot w_i(\cdot,0),\dot{\hat{w}}(\cdot,0))\tp&\in L^2(0,1)^4
}
and 
$u_e=(u_1,u_2,d)\tp\in \Lploc[2](0,\infty;\C^3)$ the closed-loop system consisting of~\eqref{eq:Wave1D} and~\eqref{eq:Wave1Dcontr} has a unique generalised solution $(x_e,u_e,y_e)$ satisfying $x_e(0)=(x_{e1}',x_{e2})\tp$. If $x_{e1}\in H^2(0,1)^4$, $x_{e2}\in H^1(0,1)^4$, and $u_e=(u_1,u_2,d)\tp\in \Hloc{1}(0,\infty;\C^3)$ satisfy the boundary conditions in~\eqref{eq:Wave1D} and~\eqref{eq:Wave1Dcontr} for $t=0$, then 
$(x_e,u_e,y_e)$ is a classical solution of the closed-loop system and $y_e\in \Hloc{1}(0,1;\C^5)$.
There exists $\gw_0<0$ such that for all $\gw\ge\gw_0$ there exist constants $M_d,M>0$ such that
every such generalised solution $(x_e,u_e,y_e)$ has the following properties:
\begin{itemize}
\item
We have
\ieq{
\displaystyle\norm{\hat d-\dtot}_{L^2_{\gw_0}(0,\infty)} \le M_d \norm{(x(0),x_s(0),x_i(0))\tp}_{X^3}.
}
\item 
If $u_e\in L_{\gw}^2(0,\infty;U_e)$, then $(y,\hat y)\tp\in L_\gw^2(0,\infty;Y^2)$ and
\eq{
\norm{x(t)} + \norm{\hat x(t)} &\le Me^{\gw_0 t} \norm{x_e(0)}_{X_e} + Me^{\gw t }\norm{(u_1,u_2)\tp}_{L_{\gw}^2(0,\infty)}\\
 \norm{(y,\hat y)\tp}_{L^2_{\gw}(0,\infty)} &\le M \norm{x_e(0)}_{X_e} + \norm{(u_1,u_2)\tp}_{L_{\gw}^2(0,\infty)}.
}
\item 
If $(u_1,u_2)\tp\in L^2(0,\infty;\C^2)$ and if 
 $d\in L^\infty(0,\infty)$
or
 $d\in L^2(0,\infty)$, 
 then 
\ieq{
\sup_{t\ge0}\; \norm{x_e(t)}_{X_e}<\infty.
}
\end{itemize}
\end{proposition}

\begin{proof}
The abstract representation of~\eqref{eq:Wave1D} has the form~\eqref{eq:plant} when we define $\phi_s=0$ and $\phi_o(y)=\phi(y_2)$ for $y=(y_1,y_2)\tp\in Y$.
With the choices
$x_s(t)=(w_s'(\cdot,t),\dot w_s(\cdot,t))\tp$,
$x_i(t)=(w_i'(\cdot,t),\dot w_i(\cdot,t))\tp$, and
$\hat x(t)=(\hat w'(\cdot,t),\dot{\hat{w}}(\cdot,t))\tp$ of state variables the controller can be represented in the form~\eqref{eq:ADRController} on $X^3=L^2(0,1)^6$ with $L_i=0$ and $L=\pmatsmall{0&-\ell \\ 0&0}\in \Lin(Y,U_b)$
 and with
$\KK\in \Lin(\Dom(\AA),U)$ defined by
$\KK x=-\kappa g(1)$ for $x=(f,g)\tp\in \Dom(\AA)$.

To apply \cref{thm:ADRCmain} we will check that the conditions in Assumption~\ref{ass:ADRCass} are satisfied. The functions $\phi_s=0$ and $\phi_o$ satisfy the required conditions and~\eqref{eq:BQCstruct} hold.
As
explained above, $(\BB,\AA,\CC,Q,B_i)$ is a well-posed boundary node on \BNspStandard. 
Integration by parts shows that 
\eqn{
\label{eq:WaveIP}
\re \iprod{\AA x}{x}_X = \re \iprod{\BB x}{\pmatsmall{0&I\\I&0}\CC x}, \qquad x\in \Dom(\AA)
}
and thus
similar arguments as in~\citel{NicPau25}{Sec.~4.1} can be used to show that $(\BB,\AA,\CC,I,0)$ is a boundary node on $\BNsp{U_b}{U_b}{X}{U_b}$. 
A direct computation shows that its transfer function $P_0$ is given by
\eq{
P_0(\gl) = \pmat{P_{dh}(\gl)& P_{dd}(\gl)\\ P_{mh}(\gl) & P_{md}(\gl)}
= \pmat{\frac{1}{\sinh(\gl)}&\frac{1}{\tanh(\gl)}\\ \frac{1}{\tanh(\gl)}&\frac{1}{\sinh(\gl)}}, \qquad \gl \in \C_0^+.
}
Since $P_0\in H_\infty(\C_\gb^+;\Lin(U_b))$ for every $\gb>0$, we have from \cref{lem:IPBCSWP}(a) that $(\BB,\AA,\CC,I,0)$ is well-posed.
The structure of $\KK$ and \cref{lem:BnodeModifications} can be used to verify that 
$(\mcB ,\mcA,  \pmatsmall{\mcC\\ \mcK}, [Q,L], [0,L_i])$ is a well-posed boundary node on $\BNsp{U_b}{U\times Y}{X}{Y\times U}$, as required.
Moreover, since $P_{dd}$ satisfies $P_{dd}(\cdot)\inv\in H_\infty(\C_\gb^+;\Lin(U))$ for any $\gb>0$, Remark~\ref{rem:Adsyssuffcond} implies that $(\BB_h-L_h\CC,\AA_d,\BB_d-L_d\CC,0,0)$ 
 is a boundary node on $\BNsp{ U_h}{\set{0}}{X}{U_d}$ with a well-posed output map.
Finally, the identity~\eqref{eq:WaveIP} and \cref{lem:IPBCSWP}(c) imply 
$(I-P_K(\cdot))\inv \in H_\infty(\C_{\gb}^+;\Lin(U)) $ and $ (I-P_L(\cdot))\inv \in H_\infty(\C_{\gb}^+;\Lin(Y))$ for some $\gb>0$.
This shows that the wave system satisfies Assumption~\ref{ass:ADRCass}.
Therefore the existence of generalised and classical solutions follows from \cref{thm:ADRCmain}. 

The semigroups $T_K$, $T_L$, and $T_d$ describe the dynamics of the one-dimen\-sion\-al wave equation with damping at one of the endpoints of the spatial interval. The semigroups $T_K$ and $T_L$ are exponentially stable by~\citel{CoxZua95}{Thm.~10.1}, and similarly $T_d$ is a contraction semigroup which can be shown to be exponentially stable using, e.g.,~\citel{ChiPau23}{Thm.~3.5} and the Gearhart--Pr\"uss--Greiner theorem. Therefore the estimates for the generalised solutions of the closed-loop system follow directly from \cref{thm:ADRCmain}.
\end{proof}

\subsection{A Heat Equation}
We consider  stabilisation for 
a one-dimensional heat equation on the interval $(0,1)$, 
\begin{subequations}
\label{eq:Heat1D}
\eqn{
\dot w(\xi,t) &= w''(\xi,t)\\
-w'(0,t) &= u(t)+d(t)+\phi(w(0,t)) , \quad  w'(1,t) = 0\\
y(t) &=  w(0,t)
}
\end{subequations}
for $\xi\in (0,1)$ and $t>0$
with initial condition $w(\cdot,0)=  w_0\in L^2(0,1)$. Here $d\in \Lploc[2](0,\infty)$ is an unknown input disturbance and the term $\phi(w(0,t))$ with an unknown globally Lipschitz function $\phi:\C\to \C$ is an unknown nonlinearity.
Our aim is to reject the unknown boundary disturbance, compensate for the nonlinearity, and achieve exponential convergence of the solutions to zero. 
Active disturbance rejection for heat equations has been previously considered in~\cite{FenGuo17d,ZhoGuo17,FenXu20}.
We may formulate the heat equation as a boundary control system of the form~\eqref{eq:plant} on $X=\Lp[2](0,1)$ with the state variable 
$x(t)=w(\cdot,t)$.
We let $U=\C$, $U_b=\C^2$, and $Y=\C$ and
define  $\AA: \Dom(\AA)\subset X\to X$ with domain $\Dom(\AA)=H^2(0,1)$ and $\BB\in \Lin(\Dom(\AA),U_b)$ and $\CC\in \Lin(\Dom(\AA),Y)$ by $\AA x=x''$, $\BB x=(x'(1),-x'(0))\tp$, $\CC x=x(0)$ for $x\in \Dom(\AA)$.
Moreover, we define $Q=\pmatsmall{0\\ 1}\in \Lin(\C,\C^2)$ and $B_i=0$. 
Then the operators $\BB$ and $\CC$ have the structure in~\eqref{eq:BQCstruct} with $U_h=\CC$, $Y_d=\C$, and $Y_m=\set{0}$ and with $\BB_h x= x'(1) $, $\BB_d x=-x'(0)$, $\CC_dx =x(0)$, and $\CC_m x=0$
for $x\in \Dom(\AA)$.
Similar analysis as in~\citel{NicPau25}{Sec.~4.1} shows that $(\BB,\AA,\CC,Q,B_i)$ is a well-posed boundary node on \BNspStandard.

Our Active Disturbance Rejection Controller has the form
\eqn{
\label{eq:Heat1Dcontr}
\begin{cases}
\dot w_s(\xi,t) = w_s''(\xi,t), \quad 
\dot w_i(\xi,t) = w_i''(\xi,t), \quad
 \dot{\hat{w}}(\xi,t) = \hat w''(\xi,t)\\
-w_s'(0,t) + \ell w_s(0,t) =  -\kappa \hat w(0,t) +w_i'(0,t)-\ell w_i(0,t)  \\
\hspace{3.9cm}+\, \ell  w(0,t)+ u_1(t)\\
w_s'(1,t) = 0, \qquad 
w_i'(1,t)  =0, \qquad
\hat w'(1,t)=0\\
 w_i(0,t) =  w(0,t) -  w_s(0,t)\\
-\hat w'(0,t) +\ell \hat w(0,t) = -\kappa \hat w(0,t)+ u_1(t) -\ell u_2(t) + \ell w(0,t)\\
\hat{d}(t) = -w_i'(0,t)+\ell w_i(0,t) \\
u(t)  = -\kappa \hat w(0,t) +w_i'(0,t)-\ell w_i(0,t)  + u_1(t).
\end{cases}
}
The closed-loop state is $x_e=(x,x_s,x_i,\hat x)\tp\in X_e=L^2(0,1)^4$ with
$x_s(t)=w_s(\cdot,t)$,
$x_i(t)=w_i(\cdot,t)$, and
$\hat x(t)=\hat w(\cdot,t)$
 and its output is $y_e=(y,\hat y)\tp =( w(0,\cdot),\hat{w}(0,\cdot))\tp$. Note that $\hat d$ is not included as an output because we will use \cref{thm:ADRCParab}, and $u_0$ may also be removed due to the fact that $u_0(t)=-\kappa \hat y(t)$.
Our main result below shows that the controller~\eqref{eq:Heat1Dcontr} stabilises the heat system~\eqref{eq:Heat1D}.

\begin{proposition}
\label{prp:Heat1Dmain}
Assume that $\phi:\C\to \C$ is globally Lipschitz continuous and let $\kappa>0$ and $\ell>0$.
For any initial conditions 
\ieq{
x_{e0}:=(w(\cdot,0),w_s(\cdot,0),$ $w_i(\cdot,0),\hat w(\cdot,0))\tp\in L^2(0,1)^4
}
and 
$u_e=(u_1,u_2,d)\tp\in \Lploc[2](0,\infty;\C^3)$ the closed-loop system consisting of~\eqref{eq:Heat1D} and~\eqref{eq:Heat1Dcontr} has a unique generalised solution $(x_e,u_e,y_e)$ satisfying $x_e(0)=x_{e0}$. If $x_{e0}\in H^2(0,1)^4$ and $u_e=(u_1,u_2,d)\tp\in \Hloc{1}(0,\infty;\C^3)$ satisfy the boundary conditions in~\eqref{eq:Heat1D} and~\eqref{eq:Heat1Dcontr} for $t=0$, then 
$(x_e,u_e,y_e)$ is a classical solution of the closed-loop system and $y_e\in \Hloc{1}(0,1;\C^2)$.
There exists $\gw_0<0$ such that for all $\gw\ge\gw_0$ there exist a  constant $M>0$ such that
every such generalised solution $(x_e,u_e,y_e)$ has the following properties:
\begin{itemize}
\item 
If $u_e\in L_{\gw}^2(0,\infty;U_e)$, then $(y,\hat y)\tp\in L_\gw^2(0,\infty;Y^2)$ and
\eq{
\norm{x(t)} + \norm{\hat x(t)} &\le Me^{\gw_0 t} \norm{x_e(0)}_{X_e} + Me^{\gw t }\norm{(u_1,u_2)\tp}_{L_{\gw}^2(0,\infty)}\\
 \norm{(y,\hat y)\tp}_{L^2_{\gw}(0,\infty)} &\le M \norm{x_e(0)}_{X_e} + \norm{(u_1,u_2)\tp}_{L_{\gw}^2(0,\infty)}.
}
\item 
If $(u_1,u_2)\tp\in L^2(0,\infty;\C^2)$ and if 
 $d\in L^\infty(0,\infty)$
or
 $d\in L^2(0,\infty)$, 
 then 
\ieq{
\sup_{t\ge0}\; \norm{x_e(t)}_{X_e}<\infty.
}
\end{itemize}
\end{proposition}

\begin{proof}
The abstract representation of~\eqref{eq:Heat1D} has the form~\eqref{eq:plant} when we define $\phi_s=0$ and $\phi_o=\phi$.
With the choices
$x_s(t)=w_s(\cdot,t)$,
$x_i(t)=w_i(\cdot,t)$, and
$\hat x(t)=\hat w(\cdot,t)$ of state variables the controller can be represented in the form~\eqref{eq:ADRController} on $X^3=L^2(0,1)^3$ 
with $L_i=0\in \Lin(Y,X)$, $L=\pmatsmall{0\\-\ell}$ with $L_h=0\in \Lin(Y,U_h)$ and $L_d=-\ell \in \Lin(Y,U)$, and with $\KK x=-\kappa \CC x=-\kappa x(0)$ for $x\in \Dom(\AA)$.
In this situation $(\BB,\AA,\CC,Q,B_i)$ is indeed a well-posed boundary node on \BNspStandard\ and~\eqref{eq:BQCstruct} hold, but the system $(\BB_d-L_d\CC,\AA_d,\BB_d-L_d\CC,0,0)$ does not have a well-posed output map. Because of this, we will employ \cref{thm:ADRCParab}.

We will first check that (a)--(c) hold in Assumption~\ref{ass:ADRCass}.
 The functions $\phi_s=0$ and $\phi_o$ satisfy the required conditions and~\eqref{eq:BQCstruct} hold, and since $L=-\ell Q$ and $\KK=-\kappa \CC$, 
$(\mcB ,\mcA,  \pmatsmall{\mcC\\ \mcK}, [Q,L], [0,L_i])$ is a well-posed boundary node on $\BNsp{U_b}{U\times Y}{X}{Y\times U}$.
Applying \cref{lem:IPBCSWP}(c) to the boundary node $(\BB_d,\AA\vert_{\ker(\BB_h)}),\CC,I,0)$ on $\BNsp{U}{ U}{X}{U}$ with $K=-\kappa I$ and $K=-\ell I$ shows that Assumption~\ref{ass:ADRCass}(c) holds with $\gb=0$.
Finally, we will show that $(\BB_1,\AA_1,\CC_1,0,0)$ in \cref{thm:ADRCParab} is a boundary node with a well-posed output map.
Similarly as in the proof of~\citel{FkiPau26arxiv}{Prop.~2.11}, the resolvent operator of $A_1:= \AA_1\vert_{\ker(\BB_1)}$ is given by
\eq{
(\gl-A_1)\inv = \pmat{(\gl-A)\inv & H(\gl) \BB_{dL} (\gl-A_d)\inv \\ 0& (\gl-A_d)\inv}
}
for $\gl\in \rho(A)\cap \rho(A_d)$, where
$H$ is the transfer function of $(\BB,\AA,\CC,Q,0)$ and
 $\BB_{dL}x=(\BB_d-L_d\CC)x=-x'(0)+\ell x(0)$ for $x\in \Dom(\AA)$.
The operator $A_d$  corresponds to the heat equation~\eqref{eq:Heat1D} with mixed homogeneous Dirichlet--Neumann boundary conditions. Thus $A$ and $A_d$ generate analytic semigroups $T$ and $T_d$, respectively.
Direct computations show that for any $\gd\in (0,\pi/2)$ we have $ \norm{H(\gl)} \lesssim \abs{\gl}^{-3/4}$ and $ \norm{\BB_{dL} (\gl-A_d)\inv}\lesssim \abs{\gl}^{-1/4} $ for $\gl\in \C\setminus \set{0}$ with $\abs{\arg{\gl}}\le \pi/2+\gd$.
This directly implies that
$\gl \mapsto \norm{H(\gl)}\norm{(\BB_d-L_d\CC)(\gl-A_d)\inv}$ is uniformly bounded on $\C_\gb^+$ for some $\gb\in \R$ and that
 $A_1$ generates an analytic semigroup on $X^2$.
 Moreover, $\BB_1$ is clearly surjective, and since its codomain is finite-dimensional, it has a bounded right inverse $\BB_1^r\in \Lin(U\times U_h,\Dom(\AA_1))$.
Thus $(\BB_1,\AA_1,\CC_1,0,0)$ is a boundary node.
Since $\KK=-\kappa \CC$, $(\BB_1,\AA_1,\CC_1,0,0)$ has a well-posed output map if  $(\BB_1,\AA_1,\CC_1',0,0)$ with $\CC_1'=[\CC,0]\in \Lin(\Dom(\AA_1),Y)$ has a well-posed output map.
The Paley--Weiner theorem implies that 
$(\BB_1,\AA_1,\CC_1',0,0)$ has a well-posed output map if for $\gw>0$ we have $\sup_{\gs>\gw}\norm{\CC_1'(\cdot+\gs-A_d)\inv x}_{L^2(i\R)}\lesssim \norm{x}$ for all $x\in \Dom(A_1)$.
We have $\norm{\CC_1'(\gl-A_d)\inv x}=\norm{\CC(\gl-A)\inv x_1+P(\gl) \BB_{dL}(\gl-A_d)\inv x_2}$ for $x=(x_1,x_2)\tp\in \Dom(A_1)$. Since $(\BB,\AA,\CC,Q,0)$ has a well-posed output map and since
for any $\gw>0$
a direct estimate   
and our earlier analysis show that $\abs{P(\gl)}\lesssim \abs{\gl}^{-1/2}$ and
$ \norm{\BB_{dL} (\gl-A_d)\inv}\lesssim \abs{\gl}^{-1/4} $ for $\gl\in \C_\gw^+$, we obtain the required estimate.
This shows that the assumptions of \cref{thm:ADRCParab} hold, and
therefore the existence of generalised and classical solutions follows from 
this result.
The operators $A_K$ and $A_L$ correspond to the heat equation~\eqref{eq:Heat1D} with modified boundary conditions $-w'(0,t)=-\kappa w(0,t)$ and $-w'(0,t)=-\ell w(0,t)$, respectively. 
Moreover, the operator $A_d$  corresponds to the heat equation~\eqref{eq:Heat1D} with mixed homogeneous Dirichlet--Neumann boundary conditions.
It is well-known that these three operators generate exponentially stable semigroups.
Therefore the estimates for the generalised solutions of the closed-loop system follow directly from \cref{thm:ADRCParab}.
\end{proof}

\subsection*{Acknowledgement} 
The authors would like to thank Nicolas Vanspranghe for helpful advice on the proof of \cref{prp:BCSNLsolutions}.

\end{document}